\documentclass[12pt, letterpaper]{article}
\usepackage[affil-it]{authblk}
\usepackage[utf8x]{inputenc}
\usepackage[T1]{fontenc}
\usepackage[pdftex]{graphicx}
\usepackage{float}

\usepackage{ucs}
\usepackage{amsmath}
\usepackage{amsfonts}
\usepackage{amssymb}
\usepackage{amsthm}

\usepackage{tikz}
\newtheorem{lem}{Lemma}
\newtheorem{thm}{Theorem}
\newtheorem{atheo}{Theorem}

\newtheorem{defin}{Definition}
\newtheorem{proposition}{Proposition}
\newtheorem{rem}{Remark}
\newtheorem{notation}{Notation}

\newcommand{\R}{\mathbb{R}}
\newcommand{\Z}{\mathbb{Z}}

\newcommand{\ho}{\hat{\Omega}}
\newcommand{\hx}{\hat{x}}
\newcommand{\hq}{\hat{q}}

\newcommand{\ds}{\displaystyle}
\usepackage{pdfpages}
\usepackage{hyperref}

\hypersetup{
    unicode=true,          % non-Latin characters in Acrobat’s bookmarks
    pdftoolbar=true,        % show Acrobat’s toolbar?
    pdfmenubar=true,        % show Acrobat’s menu?
    pdffitwindow=false,     % window fit to page when opened
    pdfstartview={FitH},    % fits the width of the page to the window
    pdftitle={ergodicity and speed up},    % title
    pdfauthor={M El Smaily and S Kirsch},     % author
    pdfsubject={Subject},   % subject of the document
    pdfcreator={Creator},   % creator of the document
    pdfproducer={Producer}, % producer of the document
    pdfkeywords={keyword1} {key2} {key3}, % list of keywords
    pdfnewwindow=true,      % links in new window
    colorlinks=true,       % false: boxed links; true: colored links
    linkcolor=black,          % color of internal links
    citecolor=red,        % color of links to bibliography
    filecolor=magenta,      % color of file links
    urlcolor=blue           % color of external links
}
\begin{document}
\title{Front speed enhancement by incompressible flows in three or higher dimensions}
\author{Mohammad El Smaily$^{\hbox{\small 1}}$\thanks{M.E. is grateful to NSERC-Canada for providing support under the NSERC postdoctoral fellowship 6790-403487.} ,  St\'ephane Kirsch$^{\hbox{\small 2}}$\\
\footnotesize{$^{\hbox{1}}$ Department of Mathematics, University of Toronto,} \\
\footnotesize{ 40 St. George Street, Toronto, ON, M5S 2E4, CANADA} \\
\footnotesize{\texttt{elsmaily@math.toronto.edu}}\\
\footnotesize{$^{\hbox{2}}$ Lyc\'ee L.G Damas,}\\
\footnotesize{Chemin Vidal, 97354 R\'emire-Montjoly, French Guiana,} \\
\footnotesize{\texttt{stephane.kirsch@ac-guyane.fr}}}

\date{July 4, 2013}
\maketitle
%\tableofcontents
\begin{abstract} We study, in dimensions $N\geq 3$, the family of first integrals of an incompressible flow: these are $H^{1}_{loc}$ functions whose level surfaces are tangent to the streamlines of the advective incompressible field. One main motivation for this study comes from earlier results proving that  the existence of nontrivial first integrals of an incompressible flow $q$ is the main key that leads to a ``linear speed up'' by a large advection of pulsating traveling fronts solving a reaction-advection-diffusion equation in a periodic heterogeneous framework. The family of first integrals is not well understood in dimensions $N\geq3$ due to the randomness of the trajectories of $q$ and this is in contrast with the case $N=2$.  By looking at the domain of propagation as a union of different components produced by the advective field, we provide more information about first integrals and we give a class of incompressible flows which exhibit `ergodic components' of positive Lebesgue measure (hence are not shear flows) and which, under certain sharp geometric conditions, speed up the KPP fronts linearly with respect to the large amplitude. In the proofs, we establish a link between incompressibility, ergodicity, first integrals, and the dimension to give a sharp condition about the asymptotic behavior of the minimal KPP speed in terms the configuration of ergodic components.\end{abstract}
%\includepdf[pages=10]{Bendixon.pdf}
\textbf{AMS subject classifications.}	35K57, 35Q35, 37A05, 37A25, 92D25.
\vskip 0.5 cm
\textbf{Keywords.} KPP front speed enhancement, propagation in a flow, reaction-advection-diffusion, ergodic components and incompressibility.

\section{Introduction} 

The main objective of this paper is to understand the influence of a strong incompressible flow on KPP reaction-diffusion  in the case where the spatial dimension is $N\geq 3.$ Consider a reaction-advection-diffusion equation of the form $$ u_t =\nabla\cdot(A(z)\nabla u)+ Mq(z)\cdot\nabla u+f(z,u),\; t\in\,\mathbb{R},\;z\in\,\Omega,$$ with boundary conditions 
        $\nu \cdot A\nabla u =0\hbox{ on }\mathbb{R}\times\partial\Omega \text { when }\partial \Omega\neq\emptyset,$
with ``standard'' assumptions on the \emph{unbounded, periodic} domain $\Omega\subseteq\R^{N}$ and the diffusion $A=A(z)$,  (see Section \ref{hetframework} for precise assumptions) and a ``KPP type'' nonlinearity (a classical example is $f(u)=u(1-u)$ for $u\in [0,1]$). This models population dynamics in a heterogeneous framework, where $u$ stands for the density of certain population at time $t$ and position $z$. The question about the influence of advection, stirring for instance, on this population dynamics is natural and has been under  investigation in many works in mathematics and physics which we will discuss in details in Subsection \ref{the-problem}.  The answer to this kind of questions in higher dimensions ($N\geq 3$) is important because there are interesting phenomena which can be described by such reaction-advection-diffusion models in the case $N=3$ particularly. In the case $N=2$, the streamlines of the incompressible flow have less freedom and this makes it relatively simpler to give a sharp criterion classifying the flows according to the rate of ``speeding-up'' the propagation phenomenon of ``traveling fronts'' induced by reaction-diffusion models.

Before discussing the heterogeneous setting, which involves a strong incompressible flow, let us first recall the notion of traveling fronts in homogeneous media and review some of their important features in the case of the so called ``KPP'' nonlinearity. Traveling fronts appeared in the pioneering work \cite{KPP} of Kolomogrov, Petrovsky and Piskunov which addressed a \emph{homogeneous} reaction-diffusion equation satisfied by a scalar quantity $u=u(t,x)$
 \begin{equation}\label{homogeneous}
 u_{t}=\Delta u+ f(u) \hbox{ for all }(t,x)\in\R\times\R^{N},
 \end{equation}
where the Lipschitz nonlinearity $f$ (often called  `KPP nonlinearity' in the literature, due to the authors of \cite{KPP}) satisfies 
\begin{equation}\label{KPPhom}
\begin{array}{ll}
f(0)=f(1)=0,~f'(0)>0, ~ f>0\hbox{ in }(0,1),\vspace{5pt}\\
0<f(s)\leq f'(0)s \hbox{ for all }s\in (0,1).
\end{array}
\end{equation}
Given a unitary direction $e\in\R^{N}$, a \emph{traveling front} in the direction of $e$ is a time-global solution to (\ref{homogeneous}) of the form $u(t,x)=\phi(x\cdot e-ct)$ 
where the \emph{profile} $\phi$ satisfies the boundary conditions $\phi(-\infty)=1$ and $\phi(+\infty)=0$. The real number $c$ is called the \emph{speed} of the front. It is well known that equation (\ref{homogeneous}) with a nonlinearity of type (\ref{KPPhom}) admits  traveling front solutions $(c,u)$, connecting the stationary states 1 and 0 (the trivial solutions of (\ref{homogeneous})) if and only if $c\geq c^{*}(e)$. The threshold $c^{*}=c^{*}(e)$ is called the \emph{KPP minimal speed}. In this simple homogeneous setting where the coefficients of the equation are independent of time and space variables and the domain is the whole space $\R^{N}$ without perforations, the value of $c^{*}$ is given by $2\sqrt{f'(0)}$ and it does not depend on the direction of propagation $e$ (see \cite{Xinreview} and the references therein). 

The interest in studying traveling front solutions and their speeds of propagation increased in the 70's due to their appearance in interface dynamics in many phenomena in chemistry and biology as well as combustion theory. For instance, the works \cite{Aronsonwein1,Aronsonwein2}, Aronson and Weinberger proved the existence of traveling wave solutions for  homogeneous reaction-diffusion equations  which were proved to model population genetics, combustion, and nerve pulse propagation. 

We want to emphasize that the minimal speed in the KPP case is the one  of special interest among all the speeds lying in the spectrum $[c^{*}(e),+\infty)$. This comes from the results of \cite{Aronsonwein2, BHNper, weinberger2002} which proved  that, in certain situations, the KPP minimal speed $c^{*}(e)$ is actually the \emph{speed of spreading} of solutions of the Cauchy problem (\ref{homogeneous}) with general compactly supported initial data.

\subsection{The heterogeneous framework}\label{hetframework}
Unlike the homogeneous setting described above, the one we work with is more complicated in the sense that the coefficients of the equation
 depend on the spacial variables and \emph{there is a drift term in the equation} as well.  The model reflects more of the realty when it  takes into account the influence of the non-homogeneous environment and medium  on the propagation phenomena. The drift term is understood in some scenarios as the representative of stirring and the main focus will be on large drifts in this work.
 A rich series of works in the last two decades discussed reaction-advection-diffusion equations in a `\emph{heterogeneous}' framework. The minimal KPP speed still exists in those settings but  it is described via elliptic eigenvalue problems related to the linearized  reaction-advection-diffusion equation near the stationary state $0.$ In this subsection, we describe the mathematical framework which we consider for this paper and we recall some important known results which are related to our analysis. The term ``reaction-advection-diffusion'' equation stands  for a model of the form
\begin{equation}\label{heteq}
 \left\{
      \begin{array}{l}
 u_t =\nabla\cdot(A(z)\nabla u)\;+\,q(z)\cdot\nabla u+f(z,u),\; t\in\,\mathbb{R},\;z\in\,\Omega,
\vspace{3pt} \\
        \nu \cdot A\nabla u =0\;\hbox{ on }\mathbb{R}\times\partial\Omega,
      \end{array}
    \right.
\end{equation}
where $\nu$ stands for the unit outward normal on $\partial\Omega$ whenever it is nonempty. In this work, we are interested in the case of large advection. That is, a parametric equation of type (\ref{heteq}) where $q\cdot\nabla u$ is replaced by $M q\cdot \nabla u$ and $M$ is a large parameter:  
\begin{equation}\label{heteqM}
 \left\{
      \begin{array}{l}
 u_t =\nabla\cdot(A(z)\nabla u)\;+\,Mq(z)\cdot\nabla u+f(z,u),\; t\in\,\mathbb{R},\;z\in\,\Omega,
\vspace{3pt} \\
        \nu \cdot A\nabla u =0\;\hbox{ on }\mathbb{R}\times\partial\Omega.
      \end{array}
    \right.
\end{equation}

In general, the domain $\Omega$ is a $C^3$ nonempty connected open subset of $\mathbb{R}^N$ such that for some integer $1\leq d\leq N,$ and for some $L_1,\cdots,L_d$ positive real numbers, we have
\begin{eqnarray}\label{comega}
    \left\{
      \begin{array}{l}
        \exists\,R\geq0\,;\forall\,(x,y)\,\in\,\Omega\subseteq\R^{d}\times\R^{N-d},\,|y|\,\leq\,R, \\
        \forall\,(k_1,\cdots,k_d)\in\,L_1\mathbb{Z}\times\cdots\,\times L_d\mathbb{Z},
        \quad\displaystyle{\Omega\;=\;\Omega+\sum^{d}_{k=1}k_ie_i},
      \end{array}
    \right.
\end{eqnarray}
where $\{e_{1}, \cdots , e_{N}\}$ stands for the standard basis of $\R^{N} $ and $d\in\{1,\cdots, N\}$. In other words, the domain $\Omega$ is $L_{i}$ periodic in the $i^{\text {th}}$ direction ($1\leq i\leq d$) and one can write $\Omega+L_{i}e_{i}=\Omega$ for all $1\leq i\leq d$. We also assume, when $d<N$, that $\Omega$ is bounded in the directions $e_i$ for $i >d$. We notice that in the case $d=N$, the domain $\Omega$ is unbounded and periodic in \emph{all directions}.  We denote the periodicity cell of $\Omega$ by
\begin{equation} \label{periodicityCell}  
C = \{z=(x_1,\cdots,x_d, y) \in \Omega \text{ such that }x_i \in [0,L_i] \text{ for all }1\leq i\leq d\}.
\end{equation}

%%\noindent The above description of $\Omega$ covers a wide variety of domains. For instance, the domain $\Omega$ may have perforations and it is \emph{unbounded}. In particular, the case $d=1$ corresponds to infinite cylinders $\Omega$ which are unbounded and periodic in one direction and which can have  periodic perforations. 

\begin{defin}[$L$-periodic fields]
On a  domain $\Omega$ which satisfies the periodicity described in  (\ref{comega}), we say that a function $g:\Omega\rightarrow \R^{m}$ ($m=1,2,\cdots$) is $L-$periodic if $g(x+k)=g(x)$ for all $x\in \Omega$ and for all $k\in L_{1}\Z\times \cdots\times L_{d}\Z\times\{0\}^{N-d}$.
\end{defin}

\noindent The presence of the term $M q\cdot \nabla u$ in equation (\ref{heteqM}), where $M$ is seen as a large parameter,  will be  the main focus of our analysis.  In any dimension $N$, the advective field $q=q(x,y)=(q_{1},\cdots,q_{N})$ is a
$C^{1,\delta}(\overline{\Omega})$ (with $\delta>0$) vector field
satisfying
\begin{equation} \label{cq}
\left\{
\begin{array}{ll}
q=q(x,y)\hbox{ is $L$- periodic in $x$}, &\vspace{3 pt}\\
  \nabla\cdot q=0 \hbox{ in } \overline{\Omega},  &\vspace{3 pt}\\
 q\cdot\nu=0\hbox{ on }\partial\Omega\hbox{ (when $\partial\Omega\neq\emptyset$)}, &\vspace{3 pt} 
 \end{array}
  \right.
\end{equation}
together with the normalization condition 
\begin{equation}\label{cq21}
\forall \, 1\leq i\leq d, ~\displaystyle{\int_{C}q_i ~dx =0}.\footnote{Proposition \ref{zeroaverageproperty} shows that this condition is equivalent to assuming that all components $\ds{\{q_{i}\}_{1\leq i\leq N}}$ are of zero average over $C$.} 
\end{equation}
%%%%%%%%%%%%%%%%%%%%%%%%%%%%%%%%%%%%
\begin{rem}\label{remarkonzeroaverage}The assumption (\ref{cq21})  on the vector field $q$ states that only the first $d$ components are of zero average over the periodicity cell $C$. The condition appeared in this form in \cite{BHNper, BHNadvec} and many other works. In fact, we will prove in Proposition \ref{zeroaverageproperty} that, due to the incompressibility of $q$, this assumption is \emph{equivalent to having all components of $q$ of zero average}. That is, 
\begin{equation}\label{cq2}
\forall \, 1\leq i\leq N, ~\displaystyle{\int_{C}q_i ~dx =0}
\end{equation}
\end{rem}

\subsection{A brief review of relevant results}\label{review}
We recall here the definition of pulsating traveling fronts and summarize the known results, which are related to our current work, regarding the existence of these fronts and their speeds in the KPP heterogeneous and periodic setting. We fix a unit direction $e\in\R^{d}$, $|e|=1$, and let $\tilde{e}:=(e,0,\cdots,0)\in\R^{N}.$ A pulsating traveling front in the direction of $e$, with a speed $c$, is a classical time-global solution $u=u(t,x,y)$ of (\ref{heteq}) which has the form 
$u(t,x,y)=\phi(x\cdot e-ct,x,y)$ for all $(x,y)\in\Omega,$  such that the `profile' $\phi$ is $L$-periodic in $x$ and connects the two stationary states of (\ref{heteq}) as follows $$\lim_{s\rightarrow-\infty}\phi(s,x,y)=1\hbox{ and }\lim_{s\rightarrow+\infty}\phi(s,x,y)=0 \text{ uniformly in } (x,y)\in\overline{\Omega}.$$
The limiting condition and the periodicity in $x$ of the profile  $\phi$ actually come from the traveling front Ansatz:  A solution $u$ to (\ref{heteq}) which satisfies 
\begin{eqnarray}
\left\{\begin{array}{l}
\displaystyle{u\left(t-\frac{k\cdot e}{c},x,y\right)=u(t,x+k,y)}  \hbox{,}\vspace{3 pt} \\
         \displaystyle{\lim_{x\cdot e \rightarrow-\infty}u(t,x,y)=1\hbox{ and } \lim_{x\cdot e\rightarrow \,+\infty} u(t,x,y)=0}  \hbox{,}
         \vspace{3 pt}\\
          0\,\leq\,u\leq\,1,
      \end{array}
    \right.
\end{eqnarray}
 where the above limits hold \emph{locally} in $t$ and \emph{uniformly} in $y$ and in the directions of $\mathbb{R}^d$ which are orthogonal to $e$ .
 
 Concerning the nonlinearity $f$ in equation (\ref{heteq}), our results will hold in the case of generalized heterogeneous  KPP type nonlinearity (not only for those of homogeneous type (\ref{KPPhom}).) In order to announce our results in the most general setting, we present these assumptions here.  The reaction term in (\ref{heteq}) is a nonnegative
function $f=f(x,y,u)$ defined in $\overline{\Omega}\,\times[0,1]$ such that
\begin{eqnarray}\label{cf1}
    \left\{
      \begin{array}{ll}
        f\geq0, f\;\hbox{ is $L$-periodic with respect to }\; x, \hbox{ and of class }C^{1,\delta}(\overline{\Omega}\times[0,1]),\vspace{3 pt}\\
        \forall\,(x,y)\in\,\overline{\Omega},\quad \displaystyle{f(x,y,0)=f(x,y,1)=0 } \hbox{,} \vspace{3 pt}\\
        \exists \,\rho\in(0,1),\;\forall(x,y)\,\in\overline{\Omega},\;\displaystyle{\forall\, 1-\rho\leq s \leq
s'\leq1,}\;
\displaystyle{f(x,y,s)\;\geq\,f(x,y,s')} \hbox{,} \vspace{3 pt}\\
        \forall\,s\in(0,1),\; \exists \,(x,y)\in\overline{\Omega}\;\hbox{ such that }\;f(x,y,s)>0  \hbox{,} \\
        \forall\,(x,y)\in\overline{\Omega},\quad \zeta(x,y):=\displaystyle{f_{u}(x,y,0)=\lim_{u\rightarrow\,0^+}\frac{f(x,y,u)}{u}>0}  \hbox{,}
      \end{array}
    \right.
\end{eqnarray}
 together with the  ``KPP'' condition (named after Kolmogorov, Petrovsky and  Piskunov \cite{KPP})
 \begin{equation}\label{cf2}
 \forall\, (x,y,s)\in\overline{\Omega}\times(0,1),~0<f(x,y,s)\leq f_u(x,y,0)\times\,s.
\end{equation}
A typical example of $f$ is $(x,y,u)\mapsto u(1-u)h(x,y)$ defined on $\overline{\Omega}\times[0,1]$ where $h$ is a positive $C^{1,\delta}(\,\overline{\Omega}\,)$ $L$-periodic function. 

\noindent Our results apply in the case of a spatially dependent diffusion  $A(x,y)=(A_{ij}(x,y))_{1\leq i,j\leq
 N}$ which is \textit{symmetric}, $C^{2,\delta}(\,\overline{\Omega}\,)$ (for some $\delta
>0$) and  satisfies the classical assumptions
\begin{eqnarray}\label{cA}
    \left\{
      \begin{array}{l}
        A\; \hbox{is $L$-periodic with respect to}\;x, \vspace{3 pt}\\
        \exists\,0<\alpha_1\leq\alpha_2,\forall(x,y)\;\in\;\Omega,\forall\,\xi\,\in\,\mathbb{R}^N,\vspace{3 pt}\\
       \displaystyle{ \alpha_1|\xi|^2 \;\leq\;\sum_{1\leq i,j\leq N}\,A_{ij}(x,y)\xi_i\xi_j\leq\alpha_2|\xi|^2.}
      \end{array}
    \right.
\end{eqnarray}
 When
$A$ is the identity matrix, this boundary condition in (\ref{heteq}) reduces to
the usual Neumann condition $\partial_{\nu}u=0.$

 Let us recall the well known existence result of KPP pulsating traveling fronts and  the threshold $c^{*}$ which,  from this point on, we denote by $c^{*}_{\Omega, A, q,f}(e)$ for the minimal KPP speed of (\ref{heteq}) in the heterogeneous setting. 

\begin{atheo}[Berestycki, Hamel, and Nadirashvili
\cite{BHNper}]\label{varthm}
 Let $e$ be a fixed unit vector in
$\mathbb{R}^d.$ Let $\tilde{e}=(e,0,\ldots,0)$ $\in\mathbb{R}^N.$
Assume that $\Omega,$ $q$, $f$ and $A$ satisfy (\ref{comega}), 
(\ref{cq}-\ref{cq21}),  (\ref{cf1}-\ref{cf2}) and (\ref{cA}). The minimal speed
$c^{*}(e):=c^{*}_{\Omega,A,q,f}(e)$ of pulsating fronts solving
(\ref{heteq}) and propagating in the direction of $e$ is given by
\begin{equation}\label{var}
    \displaystyle{c^*(e):=c^{*}_{\Omega,A,q,f}(e)\,=\,\min_{\lambda>0}\frac{k(\lambda)}{\lambda}},
\end{equation}
where $\displaystyle{k(\lambda)=k_{\Omega,e,A,q,\zeta}(\lambda)}$ is
the principal eigenvalue of the operator
$\displaystyle{L_{\Omega,e,A,q,\zeta,\lambda}}$ which is defined by
\begin{eqnarray}\label{Leq}
\begin{array}{ll}
\displaystyle{L_{\Omega,e,A,q,\zeta,\lambda}\psi\,:=}&\displaystyle{\,\nabla\cdot(A\nabla\psi)\,-2\lambda\tilde{e}\cdot
A\nabla\psi\,+q\cdot\nabla\psi\,}\vspace{3 pt}\\
&\displaystyle{+[\lambda^2\tilde{e}A\tilde{e}-\lambda\nabla\cdot(A\tilde{e})-\lambda
q\cdot\tilde{e}+\,\zeta]\psi}
\end{array}
\end{eqnarray}
acting on the set of functions
$$
\begin{array}{l}
E=\left\{\psi=\psi(x,y)\in C^2(\overline{\Omega}), \psi\hbox{ is }
L-\hbox{periodic in }x, ~
\nu\cdot A\nabla\psi=\lambda(\nu\cdot A
\tilde{e})\psi\hbox{ on }\partial{\Omega}\right\}.
\end{array}
$$
\end{atheo}

\begin{notation} \emph{We will use the notation $c^{*}_{\Omega,A, Mq,f}(e)$ for the minimal KPP speed, in the direction of $e$, of the parametric problem (\ref{heteqM}) with respect to the amplitude $M$ of the advection $q$.}
\end{notation}

\subsection{The question about the influence of an advective term $M q\cdot \nabla u$}\label{the-problem}
Our main motivation for this work comes from the very rich mathematical literature addressing the influence of a large incompressible flow on the propagation of fronts. As we mentioned above, the minimal speed is of special interest among the spectrum of speeds $[c^{*},\infty)$ in the KPP case, due to its relation to spreading of general compactly supported initial data for the Cauchy problem associated with the reaction-diffusion equation.

The presence of a large advection in  (\ref{heteqM}) is expected to speed-up the front propagation. This has been widely considered as a subject of study in many mathematical articles in the past 15 years. We firstly mention  the case of a diffusive mixing (where there is no reaction term),  
\begin{equation}\label{constantinezlatoskiselev}
u_{t}^{M}(x,t)+Mq\cdot\nabla u^{M}(x,t)-\Delta u^{M}(x,t)=0, ~~u^{M}(x,0)=u_{0}(x),
\end{equation}
 studied by  Constantin, Kiselev, Ryzhik and Zlato\v{s} \cite{ZlatosConstantineKiselevRyzhik}, who gave sharp criteria on the  incompressible flow $q$ to be ``\emph{relaxation enhancing}'' (see the precise definition in \cite{ZlatosConstantineKiselevRyzhik}).  Roughly,  a \emph{relaxation enhancing} flow is the one that makes the deviation of a solution of (\ref{constantinezlatoskiselev}) from its spatial average arbitrarily small in an arbitrarily short time $\tau$.    The criteria given in \cite{ZlatosConstantineKiselevRyzhik} were in the same spirit as those in Berestycki-Hamel-Nadirashvilli \cite{BHNadvec} which, however, were applied to study the speed up, by large advection, of KPP traveling fronts in the \emph{reaction-advection-diffusion setting}. There are important alternative ways (not via first integrals or existence of nontrivial eigenfunctions) to characterize mixing and the associated control problem which were derived by Thiffeault, Doering  \cite{thiffeaultdoering} and Lin-Thiffeault-Doering \cite{DoeringThifaultLin}. We also mention here one of the  earlier PDE works on propagation of fronts in an incompressible flow, by Majda-Souganidis \cite{MS94}, in which  the authors were able to write down and rigorously justify the appropriate renormalized effective large scale front equations for premixed turbulent combustion with two-scale incompressible velocity fields within the thermal-diffusive approximation without any ad hoc approximations.

 In our work, we deal with a large flow in the presence of diffusion and \emph{reaction}. In what follows, we will talk about the results which concern the asymptotic behavior of the, now parametric, KPP minimal speed $c^{*}_{\Omega, A, Mq, f}(e)$ when the amplitude $M$ of the advection $q$ is large. 
 
   Berestycki \cite{berestycki} and Heinze \cite{heinze} considered a particular class of incompressible flows, namely shear-flows, and proved that in any dimension $N$,  they speed-up the propagation linearly. That is, $c^{*}_{\Omega, A,Mq,f }(e)/M\rightarrow l>0$, as $M\rightarrow+\infty$, \emph{provided that}  $q$ is a \emph{shear flow} (i.e. $q(x_{1},\cdots,x_{N})=(\alpha(x_{2},\cdots,x_{N}),0\cdots,0)$ over $\R^{N}$ and $e=(1,0,\cdots,0)$). Later on, Berestycki, Hamel and Nadirashvilli \cite{BHNadvec} gave upper and lower bounds, which do not depend on the amplitude $M$, of the quantity $\ds{\frac{c^{*}_{\Omega, A, Mq,f}(e)}{M}}$ in a more general periodic framework and in the presence of a more general class of incompressible flows satisfying (\ref{cq}) and (\ref{cq2}). This heterogeneous framework is the one we described above. The upper and lower bounds of \cite{BHNadvec} affirmed that the minimal KPP speeds $c^{*}_{\Omega, Mq,f}(e)$ behave as $O(M)$ when $M$ is large; however, the precise asymptotic behavior was not given in \cite{BHNadvec}. It became interesting to know the precise limit and whether $O(M)$ is the sharp asymptotic regime of the KPP minimal speeds or not. 

One first answer to this question appeared in \cite{NovikovRyzhik} by A. Novikov and L. Ryzhik who proved that, \emph{in the case $N=2$ and for the class of cellular incompressible flows}, the parametric minimal speed $c^{*}_{ Mq,f}(e)$ behaves as $O(M^{1/4})$ when $M\rightarrow+\infty$. Later, the question ``what are all the flows which produce the sharp regime $O(M)$''  was fully answered in the \emph{2 dimensional case} by our results in \cite{ek1}. We proved that, in a general periodic framework where $N=2$,  $\ds \lim_{M\rightarrow+\infty}\frac{c^{*}_{\Omega,Mq,f}(e)}{M}$ is positive (hence $O(M)$ is the sharp regime) if and only if the advection field $q$ admits a periodic unbounded streamline. Shear flows are a particular example of these fields. To summarize, in the case $N=2$, our results in \cite{ek1} give the sharp criterion for the linear speed up and the results of Novikov and Ryzhik \cite{NovikovRyzhik} give a sharp regime ($M^{1/4}$) for the family of cellular flows which do not have unbounded periodic streamlines. It is important to mention here that \cite{BHNadvec, ekcras, ek1, heinze, ryzhikKiselev, ZlatosRyzhik, zlatosARMA} and many other works  related to the influence of large advection on KPP fronts relied on the variational formula (\ref{var}) of KPP speed of propagation - which was proved  by Berestycki, Hamel, and Nadirashvilli in \cite{BHNper} and by Weinberger \cite{weinberger2002}. The work of Constantin, Kiselev, and Ryzhik \cite{constantinekiselevRyzhik} gave several lower and upper bounds for the speeds of traveling fronts in the case of a \emph{combustion-nonlinearity or general positive nonlinearities}. In the case of  coupled reaction-advection-diffusion systems, \cite{ryzhikKiselev} gives interesting lower bounds of the parametric speeds, with respect to the amplitude of the advection, and proves that shear flows speed up the propagation in a higher rate than cellular flows.

We turn now to the precise limit as $M\rightarrow+\infty$ in any dimension which was derived in Zlato\v{s} \cite{zlatosARMA} and in our work \cite{ek1} (we recall this result in Theoem \ref{ek1} below). The speed up limits of KPP fronts by large incompressible advective fields involve a variational quantity where the functions called the ``first integrals'' of the advective field $q$ appear. These functions were used in our previous work \cite{ek1} and were previously used in Berestycki et al \cite{BHNadvec} as well as many other works in the literature (see Heinze \cite{heinze}, Ryzhik-Zlato\v{s} \cite{ZlatosRyzhik} and Zlato\v{s} \cite{zlatosARMA}).  The first integrals of $q$ are defined as follows.  
\begin{defin}[First integrals of an incompressible field \cite{BHNadvec}, \cite{ekcras}, \cite{ek1}] \label{first-integral} The set of first integrals of $q$ is defined by
\begin{equation}
\begin{array}{ll}
\mathcal{I}:=&\left\{w\in H^{1}_{loc}(\Omega),\;w \hbox{ is } L-\hbox{periodic in $x$ and }q\cdot\nabla w=0\hbox{ a.e. in }\Omega\right\}.
\end{array}
\end{equation}
Fixing a uniformly elliptic matrix $A=A(x,y)\in C^{2}(\overline{\Omega})$ satisfying (\ref{cA}), we define the following subset which relates first integrals with the reaction and diffusion terms of equation (\ref{heteqM}).
\begin{equation}\label{I1}
\mathcal{I}_1^A:=\left\{w\in\,\mathcal{I},\hbox{ such that }\int_C\zeta w^2\geq\int_C\nabla w\cdot A\nabla w\right\},
\end{equation}
where $\zeta(x,y)=f_{u}(x,y,0)$ is the positive $L$-periodic in $x$ function which we introduced in (\ref{cf2}).
\end{defin}

 \begin{atheo}[El Smaily - Kirsch \cite{ek1} and Zlato\v{s} \cite{zlatosARMA}]\label{ek1} Let $\Omega\subseteq \mathbb{R}^{N}=\R^{d}\times\R^{N-d}$ satisfy (\ref{comega}) and fix a unit direction $e\in\R^d$. Assume that the diffusion matrix $A$ and the nonlinearity $f$ satisfy (\ref{cA}), (\ref{cf1}) and (\ref{cf2}) and let $q$ be an advection field which satisfies (\ref{cq}) and (\ref{cq2}). Then,
\begin{equation}\label{large-advection}
\lim_{M\rightarrow+\infty}\ds{\frac{\ds{c^{*}_{\Omega,A,M\,q,f}(e)}}{M}}=\ds\max_{\ds{w\in \mathcal{I}_1^A}}\frac{\ds\int_C(q\cdot\tilde{e})\,w^2}{\ds\int_C w^2}.
\end{equation}
\end{atheo}
We will  work in dimensions $N\geq 3$ and give a class of flows, which are not shear flows, and which lead to a linear speed up of the KPP speed under certain conditions which turn out to be necessary and sufficient. \emph{By ``linear speed-up'' we mean that the limit (\ref{large-advection}) is strictly positive}.  

\begin{rem}\label{observation}We see from the above definition that if $w\in \mathcal I$,  then $w+\lambda\in \mathcal I$  for any constant $\lambda \in \R$. This yields that for any $w\in \mathcal I,$ there exists a sufficiently large constant $K:=K(w,C)$ so that $w+K$ belongs to $\mathcal I_{1}^{A}.$ This simple observation will be useful for us while seeking a first integral of $q$ which gives a positive limit in (\ref{large-advection}). We can now see that this limit, which is given as a maximum over $\mathcal{I}_{1}^{A}$, will be positive as long as one can find $w_{0}\in\mathcal{I}$ (not necessarily $\mathcal{I}_{1}^{A}$) such that $\int_{C}q\cdot\tilde{e} w_{0}^{2}\neq0$.
\end{rem}

\noindent We will use of the correspondence between the advection field $q$ and its `flow' in several proofs and statements. We recall this well known correspondence in the following.
\begin{defin}[Associated flow and stability of a set]
Given a vector field $q:\Omega\rightarrow\mathbb{R}^{N}$ (in this present work, $q\in C^{1,\delta}(\overline\Omega)$ and it is periodic with respect to $x$), the flow associated to $q$ or simply the flow of $q$,  is the one-parameter family of diffeomorphisms 
$\Phi:=\{\Phi_{s}\}_{s\in\R}$ generated by $q$ where $\Phi:\R\times\Omega\rightarrow\mathbb{R}^{N}$ is the unique solution of
\begin{equation}\label{flow}
\left\{\begin{array}{rl}
\ds{\frac{d}{ds}\Phi(s,x)}=&q(\Phi(s,x)), \vspace{4 pt}\\
\Phi(0,x)=&\Phi_{0}(x)=x.
\end{array}
\right.
\end{equation}

\noindent It is common to associate to $\Phi$ the one parameter family $\{\Phi_{s}\}_{s\in\R}$ where for each $s\in\R,$ $\Phi_{s}:\Omega\rightarrow \R^{N}$ is the map defined as $\Phi_{s}(x):=\Phi(s,x)$ for all $x\in\Omega$.  We recall here that for all $t,s\in\R,$ $\Phi_{s}\circ\Phi_{t}=\Phi_{t+s}$ and $\Phi_{0}=Id$.\\ In this context, a set $A\subseteq \Omega$ is said to be stable by the flow of $q$ if $\Phi_{t}(A)\subseteq A$ for all $t\in \R.$

\end{defin}
\begin{defin}[Streamlines or particle trajectories]Let $x \in \Omega$, and $\phi_x: \R \to \R^N$ be the solution of the following nonlinear ODE
$$
\left\{
\begin{array}{ccl}
\phi_x'(t) & = & q(\phi_x(t)), \\
\phi_x(0) & = & x.
\end{array}
\right.
$$
The streamline of $q$ through the point $x\in\Omega$, denoted by $T_{x}$,  is the set
\begin{equation}\label{stream}
T_x = \{\phi_x(t), \ t \in \R\}.
\end{equation}

\end{defin}

\begin{rem} \rm The streamlines of $q$ define a partition on the set $\Omega$. Notice that 
\begin{center}$x\in \Omega$ and $T_x = \{x\}$ if and only if $q(x)=0$.\end{center}
\end{rem}

%\subsection{Organization of the rest of the paper}
%Our main results are Theorems \ref{main}, \ref{H1-and-dim}, \ref{easyergodic} and \ref{ngeq5}. 

 \section{Main results: first integrals on ergodic components and speed-up of KPP fronts ($N\geq 3$)}\label{main-links}
 
 In our analysis of variational quantities of the type (\ref{large-advection}), which involve first integrals of the incompressible field $q$, it turns out  that  `\emph{ergodicity}' plays an important role \emph{in the cases $N\geq3$}. A simple way to understand this is by noticing that the condition $q\cdot \nabla w=0$, when $N=3$, means that the streamlines of $q$ are tangent to the regular level surfaces of $w$. Having $N\geq 3$ allows incompressible flows to have more degrees of freedom and this leads to more randomness in the structure of their streamlines which makes it complicated to study the level sets (surfaces) which are tangent to these trajectories. We will give the precise definition of what we call an ergodic component of an  incompressible vector field and then we prove, in Theorem \ref{main}, that over these components the first integrals of $q$ must be constant. We then apply this result to conclude about the variational quantity (\ref{large-advection}) in the case where $q$ admits ergodic components.

\noindent    \emph{Throughout this paper, we denote the Lebesgue measure on $\mathbb{R}^N$ by $\mathcal{L} ^N$}.
\begin{defin}[Ergodic components of a vector field, $N\geq 3$]\label{ergodic-comp}
Assume $N\geq3$. A set $V\subseteq \Omega$ is called an ergodic component of the vector field $q$ if $V$ is Lebesgue measurable with $\mathcal{L}^N(V)>0$, $V$ is stable by the flow of $q$ and it satisfies
$$(W \subset V \text { and } W \text{ stable by the flow of } q )\Rightarrow (\mathcal{L}^N(W)=0 \text{ or } \mathcal{L}^N(V\setminus W)=0.)$$
\end{defin}
\noindent In other words, an ergodic component in $\Omega$ produced by the advection $q$ is, in a sense,  \emph{minimal}, up to a set of measure zero, in the family of sets which are \emph{stable} by the flow associated to $q$.

\noindent It is important to know that, in the case $N=2$, such ergodic components do not exist for incompressible flows satisfying (\ref{cq}) and (\ref{cq2}). This will be proved in Appendix \ref{No-ergodicityin2d} at the end of this paper.
 
\noindent We can now state the following theorem about first integrals. \emph{The proofs of Lemma \ref{ergodic-lemma} and theorems \ref{main}, \ref{H1-and-dim}, \ref{easyergodic}, \ref{ngeq5} and \ref{shearthm} will be given in Section \ref{proofs} below.}

\begin{thm}\label{main} Let $\Omega$ be an open subset of $\mathbb{R}^N$ satisfying (\ref{comega}) (or more generally an $N$-dimensional manifold, like a flat torus).
Let $q \in C^{1,\delta}(\overline{\Omega})$ be a divergence-free vector field, and $w$ be a first integral of $q$. Then, $w$ is constant a.e. on any ergodic component of the flow.
\end{thm}
%%%%%%%%%%%%%%%%%%%%%%%%%%%%%%%%%%%%%%
\noindent The proof of Theorem \ref{main} relies on the following lemma which  holds in any dimension $N$.
%%%%%%%%%%%%%%%%%%%%%%%%%%%%%%%%%%%%%%%% 
\begin{lem}\label{ergodic-lemma} Assume that $\Omega\subseteq\R^{N}$ is an open connected domain which satisfies (\ref{comega}). Let $w$ be a first integral of $q$ on $\Omega$ and $I$ a measurable subset of $\mathbb{R}$. Then, up to a set of measure $0$, $w^{-1}(I)$ is stable by the flow of $q$. Furthermore,
$$
\forall t \in \mathbb{R}, \quad \mathcal{L}^{N}\left(\Phi_t(w^{-1}(I))\Delta \left(w^{-1}(I)\right)\right)=0,
$$
where $\Delta$ stands for the symmetric difference\footnote{for any two sets $A$ and $B$, $A\Delta B$ stands for the $(A\setminus B) \cup (B\setminus A)=(A\cup B)\setminus (A \cap B)$}and $\Phi$ is the flow associated to $q$.
\end{lem}
\begin{rem} We emphasize here that the result of Lemma \ref{ergodic-lemma} is valid in any dimension $N$---not only in dimensions smaller than or equal to 3. It is also worth mentioning that assumptions (\ref{ergodic-lemma}) on the domain $\Omega$ are not all necessary for the result to hold (It will be easy to see this throughout the proof of the lemma), but we use these assumptions to guarantee the existence of traveling fronts which are our main motivation for this study. \end{rem}

\subsection{Impact of the configuration of ergodic components on first integrals}\label{configuration}
In Theorem \ref{main} above, we established a link between ergodicity and first integrals.   The following theorem aims to show the influence of the dimension on the $H^{1}$ regularity of a function which admits two different constant values over two balls which are tangent to each other (in the next results, these functions will play the role of first integrals to the flow). We will  use the next theorem to study the particular class of vector fields having ergodic components of positive Lebesgue measure and investigate whether they can give a  linear speed up of the KPP speed $c^{*}$ or not. 

\begin{thm}\label{H1-and-dim}
Let $N \in \mathbb{N}$ such that $N \geq 2$, $B_1$ the open ball in $\mathbb{R}^N$ of radius 1 and center $(0,\cdots, 0,1)$, $B_2$ the open ball of radius 1 and center $(0,\cdots,0,-1)$ and $U$ a bounded open subset of $\mathbb{R}^N$ containing the convex hull of $B_1 \cup B_2$.\\ For a function $u: U \rightarrow\mathbb{R}$ and for $(\lambda,\mu) \in \mathbb{R}^2$ be any couple, we say that $u$ verifies  $(C_{\lambda,\mu})$ if 
\begin{equation*}
(C_{\lambda,\mu}) \qquad \qquad \qquad \qquad \qquad \qquad u|_{B_1}=\lambda \text { and } u|_{B_2}=\mu. \qquad \qquad \qquad \qquad \qquad \qquad
\end{equation*}
Depending on the dimension $N$, we have the following
\begin{enumerate}
\item If $N \leq 3$, and if $u\in H^{1}(U)$ verifies $(C_{\lambda,\mu})$, we must have $\lambda=\mu$.
\item If $N \geq 4$, then for any couple $(\lambda,\mu)\in \mathbb{R}^{2}$, there exists a function $u \in H^1(U)$ verifying $(C_{\lambda,\mu})$.
\end{enumerate}

\end{thm}
%%%%%%%%%%%%%%%%%%%%%%%%%%%%

%%%%%%%%%%%%%%%%%%%%%%%%%%

\begin{rem}\label{N=3-is-the-most}
We can  see now  that the existence of nontrivial  first integrals is \emph{particularly} subtle  in the case $N=3$:   First, we know that  incompressible fields satisfying (\ref{cq}-\ref{cq21}) and having ergodic components (in the sense of Definition \ref{ergodic-comp}) exist in the case $N\geq 3$ but not in the case $N=2$. On the other hand, Theorem \ref{main} and Theorem \ref{H1-and-dim} allow us to see that the $H^{1}_{loc}$-regularity of first integrals, for incompressible flows with ergodic components, is strongly affected by the configuration of these components in the case $N\leq3$. This will become more clear in the following results.
\end{rem}
%%%%%%%%%%%%%%%%%%

%%%%%%%%%%%%%%%%%%%

\subsection{The KPP speed in a large advective field with ergodic components ($N\geq 3$)}

We will apply the results of Theorem \ref{main} and Theorem \ref{H1-and-dim} to give a class of incompressible flows, \emph{other than shear flows}, which make the limit (\ref{large-advection}) strictly positive \emph{in the case where the dimension is $N=3$ or higher}.   Our result in \cite{ek1} and the results of Novikov-Ryzhik \cite{NovikovRyzhik} show that an efficient way to study the influence of a large incompressible flow on the reactive-diffusive front is by looking at the components produced by this flow inside the domain of propagation. This simple observation led, in the 2 dimensional case, to the sharp criterion which roughly states: when $N=2$, the limit (\ref{large-advection}) is positive (i.e. the advection speeds up the KPP fronts linearly) if and only if it admits an unbounded periodic streamline (see \cite{ek1} for the proof).

 We defined ergodic components produced by an incompressible flow in Definition \ref{ergodic-comp} above.  We sketch here the strategy which we will use to handle the variational quantity (\ref{large-advection}) in the cases $N\geq 3$. For simplicity, suppose that $\Omega=\R^{N}$ and that we have a smooth incompressible flow $v$ over  $\Omega$ (for the existence of such $v$, see below) such that:
\begin{itemize}
\item[(A1)] $v=v(x_{1},\cdots,x_{N})=(v_{1},\cdots,v_{N})$  is periodic in $x_{1},\cdots, x_{N}$. 
 \item[(A2)] $v$ admits an ergodic component $V_{1}=\R \times D_{1} \subset \R^{N}$ ($D_{1}$ is a ball in $\R^{N-1}$) which is a cylinder in the direction $e=(1,0,\cdots,0)$.
 \item[(A3)] $v\equiv0$ on $(\R\times \partial D_{1})\cup (\R^{N}\setminus V_{1})=\partial V_{1}\cup (\R^{N}\setminus V_{1})=\overline{V_{1}^{c}}$,
\end{itemize}
then the set of first integrals $\mathcal {I}$ of $v$ contains only the $H^{1}_{loc}$ functions which are \emph{almost everywhere constant over the component $V_{1}$.} Theorems \ref{main} and \ref{H1-and-dim} will then be useful to give the answer to the question about the positivity of the limit (\ref{large-advection}) which is indeed the linear speed up of KPP fronts.

 \emph{Existence results of a flow $v$ satisfying (A1)-(A2)-(A3), which is incompressible, periodic and exhibit ergodic components, were proved in  details  by H.~Hu, Y.~Pesin, and A.~Talitskaya in \cite{hpt2004} which  followed a study done by Katok \cite{katok}.  The result of Pesin \emph{et al} \cite{hpt2004} states that  ``\emph{every compact manifold carries a hyperbolic ergodic flow}'' provided that the dimension of the manifold is \emph{greater or equal 3.}}

In our setting, we have a periodic structure in the domain $\Omega\subseteq\R^{N}$, and therefore, we can apply the results of Hu, Pesin and Talitskaya \cite{hpt2004} on the whole periodic set $\Omega$ and get the flow $v$. After a normalization to a zero-average flow $q$ (see next paragraph), we will have a vector field $q$ which satisfies all the properties to be considered as particular example of incompressible flows with ergodic components $V_{i}\subseteq \Omega$.

We work with advection fields which have zero average and admit ergodic components in the same direction. To guarantee that the advection  $q$ (in equation (\ref{heteqM})) is of zero average, consider 2 cylinders which are aligned in the direction of $e=(1,0,\cdots,0)$ denoted  by  $$V_{1}:=\R\times D_{1} \text{ and } V_{2}:=\R\times D_{2},$$ where $D_{1}$ and $D_{2}$ are two open balls in $\R^{N-1}$ having the \emph{same radius} $R$ and centered at $O_{1}(0,\cdots,0,a+2R+h)$ (for some $h\geq 0$)  and $O_{2}(0,\cdots,0,a)$ respectively  and such that $[0,L_{1}]\times (D_{1}\cup D_{2})\subseteq C$. The number $h\geq0$ is  the distance between $\partial V_{1}$ and $\partial V_{2}$ and the centers of $D_{1}$ and $D_{2}$ are at a distance $2R+h$ from each other.  
\begin{figure}[t]
  \centering
 \begin{center}
\begin{tikzpicture}[scale=.7]
\shade[top color=gray!50, bottom color=white] (.5,2.5) arc (0:90:.5cm and 1.5cm) -- (8,4) arc (90:0:.5cm and 1.5cm) --cycle;
\shade[top color=white, bottom color=gray!50] (.5,2.5) arc (0:-90:.5cm and 1.5cm) -- (8,1) arc (-90:0:.5cm and 1.5cm) --cycle;
\shade[draw,outer color=gray!50, inner color=white,very thick] (0,2.5) ellipse  (.5cm and 1.5cm);
\draw[very thick] (0,4)--(8,4) arc (90:-90:.5cm and 1.5cm) -- (0,1);
\shade[top color=gray!50, bottom color=white] (.5,-2.5) arc (0:90:.5cm and 1.5cm) -- (8,-1) arc (90:0:.5cm and 1.5cm) --cycle;
\shade[top color=white, bottom color=gray!50] (.5,-2.5) arc (0:-90:.5cm and 1.5cm) -- (8,-4) arc (-90:0:.5cm and 1.5cm) --cycle;
\shade[draw,outer color=gray!50, inner color=white,very thick] (0,-2.5) ellipse  (.5cm and 1.5cm);
\draw[very thick] (0,-1)--(8,-1) arc (90:-90:.5cm and 1.5cm) -- (0,-4);
\draw [very thick, dashed] (8,4) arc (90:270:.5cm and 1.5cm);
\draw [very thick, dashed] (8,-4) arc (-90:-270:.5cm and 1.5cm);
\draw[<->,very thick] (4,-.9)--(4,.9) node[midway,right]{\Large $h$};
\draw (3.5,0) node[left]{\Large $q=0$};
\draw (2.5,2.5) node[left]{\Large $v$};
\draw (2.5,-2.5) node[left]{\Large $-v$};
\draw (-1,2.5) node[left]{\Large $V_1$};
\draw (-1,-2.5) node[left]{\Large $V_2$};
\end{tikzpicture}
\end{center}
\caption{$N=3$, two cylindrical ergodic components of $q$ aligned in the same direction with a gap of height $h\geq0$ in between.}
  \label{ergcylinders}
\end{figure}
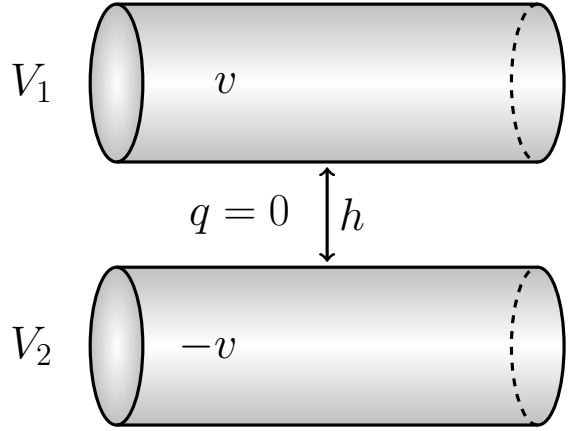
Then, define the vector field $q$ over $\R^{N}$  (see Figure \ref{ergcylinders}) by
\begin{equation}\label{qcylinders}
\left\{\begin{array}{ll}
q(x)=v(x) \text{ for all } x\in V_{1}, \vspace{5 pt}\\ \text{where $v$ is the vector field constructed in (A1)-(A3) above},\vspace{5 pt}\\
q(x_{1},\cdots,x_{N})= -\,v(x_{1},\cdots,x_{N-1},-x_{N}+2R+h+a), ~ x\in V_{2},\vspace{5 pt}\\
q(x)=0, ~x\in \overline{\Omega\setminus (V_{1}\cup V_{2})}\supseteq(\R\times \partial D_{1})\cup (\R\times \partial D_{2}).
\end{array}\right.
\end{equation}
The vector field $q$ in (\ref{qcylinders}) is then smooth, incompressible, periodic, admits two ergodic components $V_{1}$ and $V_{2}$ in the direction of $e$, vanishes on $\partial V_{1}$, $\partial V_{2}$ and in $\Omega \setminus (V_{1}\cup V_{2})$ and  satisfies $$\int_{C}q dx=\int_{[0,L_{1}] \times D_{1}}v(x)dx-\int_{[0,L_{1}]\times D_{2}}v(x_{1},x_{2},-x_{3}+2R+h)dx=\vec{0}.$$ 

\noindent We can now state the following theorem in the case $N=3$ which is also true in the case $N=4$\footnote{See Remark \ref{n=4} below the proof of Theorem \ref{easyergodic} about the case $N=4$.}.
\begin{thm}\label{easyergodic}
Let $\Omega=\R^{3}$ or $\R\times\omega$ where $\omega\subseteq\R^{2}$ satisfies (\ref{comega}) and let $V_{1}:=\R\times D_{1}$ and $V_{2}:=\R\times D_{2}$ be the cylindrical subsets of $\Omega$, defined above, in the same direction $e=(1,0,0)$. Let $q$ be a 3 dimensional incompressible flow of type (\ref{qcylinders}) which has a  zero-average over $C$ and admits two ergodic components $V_{1}$ and $V_{2}$. Consider the reaction-advection-diffusion equation (\ref{heteqM}) with this particular advection $q$, where $f$ and $A$ satisfy  (\ref{cf1}-\ref{cf2}) and (\ref{cA}) respectively. Then, the minimal KPP speed $c^{*}_{\Omega,A, Mq,f}(e)$ satisfies
\begin{equation}\label{positivespeedup3}
0<\lim_{M\rightarrow+\infty}\frac{c^{*}_{\Omega,A, Mq,f}(e)}{M}<\infty \text{ if and only if $h:=dist(\overline{V_{1}},\overline{V_{2}})>0$.}
\end{equation} 
In particular, if the ergodic components $V_{1}$ and $V_{2}$ of $q$ are tangent to each other, we then have 
$$\lim_{M\rightarrow+\infty}\frac{c^{*}_{\Omega,A, Mq,f}(e)}{M}=0.$$
\end{thm}

\begin{rem} 
The number of ergodic components in the above theorem, and in Theorem \ref{ngeq5} below, does not have to be exactly two for the result to hold. This was just added to simplify the construction a vector field of zero average. The above theorem, as well as Theorem \ref{ngeq5}, hold true in the case of an incompressible flow $q$ with a countable collection of ergodic components $\{V_{i}\}_{i\in \mathbb{N}}$ provided that $q$ vanishes on the boundary of each $V_{i}$ and $\int_{C}q =0$. One can easily see from the computation in (\ref{leading}).
\end{rem}
We end this subsection by a result about the asymptotic behavior of $c^{*}_{\Omega,A, Mq, f}(e)$, in presence of large advection with ergodic components, in dimensions $N\geq 5$.  
 The difference between the case $N\geq 5$ and the case $N=3$ or $4$ is that the limit (\ref{large-advection}) will be always positive regardless of the distance $h\geq0$ between the ergodic components. That is, the limit (\ref{large-advection}) will be positive, in $N\geq 5$, even in the case where $h=0$.
\begin{thm}\label{ngeq5}
Assume that $N\geq 5$ and let $\Omega=\R^{N}$ or $\R\times\omega$ where $\omega\subseteq\R^{N-1}$ satisfies (\ref{comega}). Let $V_{1}$ and $V_{2}$ be  the cylindrical subsets of $\Omega$, as defined above, in the same direction $e=(1,0,\cdots,0)$, such that $[0,L_{1}]\times (D_{1}\cup D_{2})\subseteq C$ ($C$ is the periodicity cell of $\Omega$). Consider an $N$-dimensional incompressible flow $q$ of type (\ref{qcylinders}) which has a  zero-average over $C$ and admits two ergodic components $V_{1}$ and $V_{2}$. Consider the reaction-advection-diffusion equation (\ref{heteqM}) with this particular advection $q$, where $f$ and $A$ satisfy  (\ref{cf1}-\ref{cf2}) and (\ref{cA}) respectively. Then, the minimal KPP speed $c^{*}_{\Omega,A, Mq,f}(e)$ always satisfies
\begin{equation}\label{positivespeedupngeq5}
0<\lim_{M\rightarrow+\infty}\frac{c^{*}_{\Omega,A, Mq,f}(e)}{M}<\infty.
\end{equation} 
 \end{thm}

\subsection{A comparison between the influence of flows with ergodic components to that of shear flows}
We will compare the above results of Theorem \ref{easyergodic}, where the advection $q$ admits ergodic components, to the case of shear flows which always lead to a linear speed up of the KPP minimal speed. In the 3 dimensional setting, unlike the flows with ergodic components (see Theorem \ref{easyergodic} above), and with a remarkable contrast, the limit $\lim_{M\rightarrow+\infty}c^{*}_{Mq}(e)/M $ is always positive \emph{when  $q$ is a shear flow in the direction of $e$}. This is precisely Theorem \ref{shearthm} which appeared in  Berestycki \cite{berestycki} and Heinze \cite{heinze}. For the reader's convenience, we will present a short proof of this theorem by using our result in \cite{ek1} or the one in Zlato\v{s} \cite{zlatosARMA}.

\begin{thm}[\cite{berestycki} and \cite{heinze}]\label{shearthm}
Let $N\geq 2$ and assume that the domain has the form $\Omega:=\mathbb{R}\times \omega$ where $\omega\subseteq \R^{d}\times \R^{N-1-d}\subseteq \R^{N-1}$ satisfies (\ref{comega}), or $\omega$ is a bounded smooth open subset of $\R^{N-1}$ (in which case $d=0$), that the direction of propagation is $e=(1,0,\cdots,0)\in  \mathbb{R}^{N}$ and let $q$ be a shear flow of the form 
\begin{equation}\label{shearflow}
q(x):=(q_{1}(x_{2},\cdots,x_{N}),0,\cdots,0)
\end{equation} 
with $q_{1}\in C^{1,\delta}(\overline{\Omega})$, $q_{1}\not \equiv 0$ and $\int_{C}q_{1}=0$. Then,
\begin{equation}\label{shearedspeed}
\lim_{M\rightarrow+\infty}\frac{c^{*}_{\Omega, A, Mq, f}(e)}{M}=\max_{\psi\in \mathcal{J}}\frac{\int_{C}q_{1}(x_2,\cdots,x_{N})\psi^{2}}{\int_{C}(\psi(x_{2},\cdots,x_{N}))^{2}dx}>0,
\end{equation} 
where $$\mathcal{J}:=\{\psi=\psi(x_{2},\cdots,x_{N})\in H^{1}_{loc}(\omega), ~\psi \text{ is periodic in the unbounded directions of }\omega \}$$
\end{thm}
 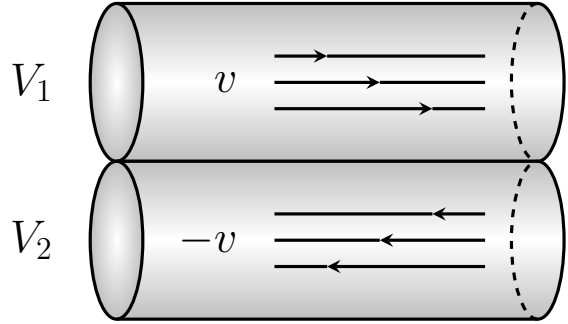
\begin{figure}[H]
\centering
\begin{center}
\begin{tikzpicture}[scale=0.7]
\shade[top color=gray!50, bottom color=white] (.5,1.5) arc (0:90:.5cm and 1.5cm) -- (8,3) arc (90:0:.5cm and 1.5cm) --cycle;
\shade[top color=white, bottom color=gray!50] (.5,1.5) arc (0:-90:.5cm and 1.5cm) -- (8,0) arc (-90:0:.5cm and 1.5cm) --cycle;
\shade[draw,outer color=gray!50, inner color=white,very thick] (0,1.5) ellipse  (.5cm and 1.5cm);
\draw[very thick] (0,3)--(8,3) arc (90:-90:.5cm and 1.5cm) -- (0,0);
\shade[top color=gray!50, bottom color=white] (.5,-1.5) arc (0:90:.5cm and 1.5cm) -- (8,0) arc (90:0:.5cm and 1.5cm) --cycle;
\shade[top color=white, bottom color=gray!50] (.5,-1.5) arc (0:-90:.5cm and 1.5cm) -- (8,-3) arc (-90:0:.5cm and 1.5cm) --cycle;
\shade[draw,outer color=gray!50, inner color=white,very thick] (0,-1.5) ellipse  (.5cm and 1.5cm);
\draw[very thick] (0,0)--(8,0) arc (90:-90:.5cm and 1.5cm) -- (0,-3);
\draw [very thick, dashed] (8,3) arc (90:270:.5cm and 1.5cm);
\draw [very thick, dashed] (8,-3) arc (-90:-270:.5cm and 1.5cm);
\draw (2.5,1.5) node[left]{\Large $v$};
\draw (2.5,-1.5) node[left]{\Large $-v$};
\draw (-1,1.5) node[left]{\Large $V_1$};
\draw (-1,-1.5) node[left]{\Large $V_2$};
\draw [>=stealth, ->,very thick] (3,2)--(4,2);
\draw [very thick] (4,2)--(7,2);
\draw [>=stealth, ->,very thick] (3,1.5)--(5,1.5);
\draw [very thick] (5,1.5)--(7,1.5);
\draw [>=stealth, ->,very thick] (3,1)--(6,1);
\draw [very thick] (6,1)--(7,1);
\draw [>=stealth, ->,very thick] (7,-2)--(4,-2);
\draw [very thick] (4,-2)--(3,-2);
\draw [>=stealth, ->,very thick] (7,-1.5)--(5,-1.5);
\draw [very thick] (5,-1.5)--(3,-1.5);
\draw [>=stealth, ->,very thick] (7,-1)--(6,-1);
\draw [very thick] (6,-1)--(3,-1);
\end{tikzpicture}
\end{center}
\caption{\footnotesize{A shear flow, over one periodicity cell $C$, having two cylindrical components aligned in the same direction and tangent to each other. Over $\overline{V_{1}}\cap \overline{V_{2}},~q=0$.}}
\label{shearfigure}
\end{figure}
\begin{rem}It is important to notice the remarkable difference between influence of shear flows and the ergodic ones on the KPP speed. In Theorem \ref{shearthm}, $N=3$, we see that shear flows make the limit (\ref{shearedspeed}) positive regardless of the configuration of the flow over its components. To be more precise, the shear flow $q$ could have, over one periodicity cell, two components $V_{1}$ and $V_{2}$ which are tangent to each other (as in Figure \ref{shearfigure})), where $q$ vanishes on their boundaries $\partial V_{1}$ and $\partial V_{2}$, and yet, the speed up of the KPP pulsating fronts stays linear (i.e. (\ref{shearedspeed}) holds true). While, as  we saw in Theorem \ref{easyergodic} above, this cannot be the case for flows with ergodic components, when $N=3$ (the speed up is linear with respect to the amplitude of $q$ if and only if the ergodic components are at a distance from each other). 
\end{rem}

%%%%%\section{Conclusions}
%%%\textbf{We proved that in the 3 dimensional case, a class of incompressible flows, different from Shear flows, could linearly speed up the propagation as long as there is a gap between its ergodic components. We demonstrated a link between ergodicity and incompressibility to prove that first integrals have to be constant a.e. on the ergodic components of the flow. Seeing these flows in contrast with Shear flows is one of the interesting new points in the case $N\geq 3.$ These ideas will help to get a sharp criterion on advection fields ........... }

\section{Proofs}\label{proofs}
%%%%%%%%%%%%%%%%%%%%%%%%%%%%%%%%%%%%%%%%%%%%%%%%%%%%%%
In this section, we prove Lemma \ref{ergodic-lemma} and then the main results: Theorem \ref{main}, Theorem \ref{H1-and-dim}, Theorem \ref{easyergodic} and Theorem \ref{ngeq5}. We also give a short proof of Theorem \ref{shearthm} which we reviewed in this work to show the contrast between ergodic flows and shear flows in dimensions $N\geq 3$.

%%%%%%%%%%%%%%%%%%%%%%%%%%%%%%
\subsubsection*{Proof of Lemma \ref{ergodic-lemma}}
 Let $v \in C^\infty_b(\Omega)$ (smooth and bounded function), $I$ be a measurable subset of $\mathbb{R}$ and let $\Phi$ denote the flow associated to $q$. Then, for all $t$,
\begin{eqnarray}\label{6thineq}
\nonumber \int_{\Phi_t(v^{-1}(I))} |v(x)-v(\Phi_{-t}(x))|dx  & = & \int_{\Phi_t(v^{-1}(I))} \left|\int_{-t}^0 \dfrac{\partial }{\partial s}
\left(v(\Phi_s(x))\right) ds\right|dx  \vspace{5 pt}\\
\nonumber & \leq & \int_{\Phi_t(v^{-1}(I))} \left|\int_{-t}^0 \left|\dfrac{\partial }{\partial s}
\left(v(\Phi_s(x))\right)\right| ds\right|dx \vspace{5 pt}\\
\nonumber& \leq & \int_{\Phi_t(v^{-1}(I))} \int_{[-t,0]} \left|q(\Phi_s(x)) \cdot \nabla v(\Phi_s(x))\right| dsdx \vspace{5 pt}\\
\nonumber (\hbox{by Fubini's theorem})& \leq &   \int_{[-t,0]} \int_{\Phi_t(v^{-1}(I))} \left|q(\Phi_s(x)) \cdot \nabla v(\Phi_s(x))\right| dxds \vspace{5 pt}\\
\nonumber & \leq & \int_{[-t,0]} \int_{\Omega} \left|q(\Phi_s(x)) \cdot \nabla v(\Phi_s(x))\right| dxds\\  \vspace{5 pt}&
 \leq & \int_{-t}^0\int_\Omega \left|q(y) \cdot \nabla v(y)\right| dy ds. 
\end{eqnarray}
In (\ref{6thineq}), we used the change of variable $y=\Phi_s(x)=\Phi(s,x)$, and since $q$ is incompressible, the Jacobian $J(s,x):=\det \left[\nabla_{x} \Phi(s,x)\right]$ of this change of variable is equal to 1 (at any $s$ and any $x\in\Omega$). We refer the reader to Proposition 1.3 and Proposition 1.4 in Majda and Bertozzi \cite{BertozziMajda} for a proof of this fact.\footnote{For a fixed $x$, the Jacobian $J(x,t):=\det \left[\nabla_{x}\Phi(t,x)\right]$ satisfies
the differential equation $\ds{\frac{\partial J}{\partial t}(x,t)=(\nabla_{x}\cdot  q)}\ds|_{_{\ds\Phi(t,x)}} J(x,t)$ with the initial condition $J(x,0)=\ds{ \det\left[\frac{\partial \Phi_{0}}{\partial x}\right]}=\ds{\det\left[\ds{Id_{\ds M_{N}(\R)}}\right]}=1$ (recall that $\Phi_{0}(z)=z$ for all $z\in\Omega$). As $\nabla\cdot q\equiv0,$  it follows that $\partial_{t}J=0$ and hence $J(x,\cdot)$ is constant as a function of $t\in\R$ and it must be equal to $J(x,0)=1$.} So far, we have
\begin{eqnarray}\label{conservation-estimate}
\nonumber \forall  v\in C_{b}^{\infty}(\Omega),\\
 \ds \int_{\Phi_t(v^{-1}(I))} \left|v(x)-v(\Phi_{-t}(x))\right|dx & \leq & |t|\int_\Omega \left|q(y) \cdot \nabla v(y)\right| dx .
\end{eqnarray}
By density,  inequality (\ref{conservation-estimate}) remains true for all $v \in H_{loc}^1(\Omega)$. In particular, if $w$ is a first integral of $q$ (see Definition \ref{first-integral} above), (\ref{conservation-estimate}) becomes
\begin{equation} \label{eq01}
\forall t\in\R, \ \int_{ \Phi_t(w^{-1}(I))}  \left|w(x)-w(\Phi_{-t}(x))\right|dx = 0.
\end{equation}
Moreover, if $x \in \Phi_t(w^{-1}(I))$,  $\Phi_{-t}(x) \in w^{-1}(I)$; thus, $w(\Phi_{-t}(x)) \in I$.
From (\ref{eq01}) we have, for almost every $z \in \Phi_t(w^{-1}(I))$, $w(z) =w(\Phi_{-t}(z))$. 
%Therefore,
%$$
%\hbox{for almost every }x\in\Phi_t(w^{-1}(I)), ~~w(x)  \in I \Leftrightarrow x \in %w^{-1}(I).$$
This yields that, for almost every $z \in w^{-1}(I)$, $\Phi_t(z) \in w^{-1}(I)$ for all $t\in \mathbb{R}$. This allows us to conclude that, up to a set of measure 0,
$$
\forall t\in\R,\ \Phi_t(w^{-1}(I)) \subseteq w^{-1}(I).
$$
The previous inclusion and the fact that the flow is measure preserving (as explained above in the case of incompressible fields)
lead us to conclude that, up to a set of measure $0$,
$\Phi_t(w^{-1}(I)) = w^{-1}(I).$
In other words, $$\mathcal{L}^{N}\left(\Phi_t(w^{-1}(I)) \Delta (w^{-1}(I))\right)=0,\ \hbox{ for all }t\in\R,$$ and this completes the proof of Lemma \ref{ergodic-lemma}. \hfill $\Box$
\vskip 0.25cm
%%%%%%%%%%%%%%%%%%%%%%%%%%%%%%%%%%%%%%%%%%%%%%%%%%%%%%%%%%%%%%%%%%%%%

\noindent We are now in the position to prove Theorem \ref{main}.

\subsubsection*{Proof of Theorem \ref{main}}
Let $V$ be an ergodic component of the flow, and $w$ a first integral of $q$. We consider the following function:
\begin{eqnarray*}
f:\mathbb{R} & \longrightarrow & \mathbb{R} \\
t & \longmapsto & \mathcal{L}^N\left(w^{-1}((-\infty,t])\cap V\right)
\end{eqnarray*}
$f$ is an upper semi-continuous increasing function which satisfies
$$
\lim_{t\rightarrow-\infty} f(t) = 0 \qquad \text{ and } \qquad \lim_{t\rightarrow+\infty} f(t) = \mathcal{L}^N(V).
$$
We assume to the contrary that $f(\mathbb{R}) \neq \{0,\mathcal{L}^N(V)\}$. We can then pick $\lambda \in f(\mathbb{R})\backslash \{0,\mathcal{L}^N(V)\}$ and $t_0$ such that $f(t_0)=\lambda$. We set $W=w^{-1}((-\infty,t_0])\cap V$. By definition of $f$ and $\lambda$, we get
$$
0<\mathcal{L}^N(W)=\lambda<\mathcal{L}^N(V).
$$
Moreover, $W$ is the intersection of $V$ (which is stable by the flow) with $w^{-1}((-\infty,t_0])$ which, by Lemma 1, is also stable by the flow, up to a set of measure 0. Hence, $W$ is stable by the flow up to a set of measure 0. However, this  contradicts the ergodicity of $V$. We can now conclude that $f(\mathbb{R})=\{0,\mathcal{L}^N(V)\}.$ Let $$\alpha = \inf\{t \in \mathbb{R}, \ f(t)=\mathcal{L}^N(V) \}.$$ One can see that, up to a set of measure 0, $V=V\cap w^{-1}(\{\alpha\})$, and therefore,
$w(x)=\alpha$, for almost every $x \in V$. This completes the proof of the theorem.\hfill $\Box$

\subsubsection*{Proof of Theorem \ref{H1-and-dim}} 

1. We assume to the contrary that there exists $(\lambda,\mu)\in \mathbb{R}^2$ with $\lambda\neq \mu$ and $u \in H^1(U)$ verifying condition $(C_{\lambda,\mu})$.
Let $\rho \in C^{\infty}_c(\mathbb{R}^N)$, such that $\rho \geq 0$, $\rho$ has support in $D^N$ the open unit ball in $\mathbb{R}^N$  and $\int_{\mathbb{R}^N}\rho=1$.\\
For $n \in \mathbb{N}^*$, we set $\rho_n:x \mapsto n\rho\left( nx \right)$. $\rho_n$ is a mollifier with support in the open ball of center $0$ and radius $\dfrac{1}{n}$.\\
Let $$U_n=\left\{x \in \mathbb{R}^N, \ dist(x,U)\leq \dfrac{1}{n} \right\}.$$

Let $u_n = u\star \rho_n$, where we extend $u$ by $0$ outside $U$. Then, $u_n \in C^{\infty}_c(U_n)$ and
$$
\ds{u_n|\ds{_{_U}}} \xrightarrow[n\to +\infty]{H^1(U)} u.
$$
We denote by $B_{1,n}$ the ball of $\mathbb{R}^N$ of center $(0,...,0,1)$ and radius $1-1/n$ and by $B_{2,n}$ the ball of $\mathbb{R}^N$ of center $(0,...,0,-1)$ and radius $1-{1}/{n}$. Then, $u_n|_{B_{1,n}}=\lambda$ and $u_n|_{B_{2,n}}=\mu$ (see Figure \ref{fig1} in the case $N=3$.)

We now use cylindrical coordinates in $\mathbb{R}^N$. 
That is, if  $(x_1,...,x_N)$ are the cartesian coordinates of $x \in \mathbb{R}^N$, we denote by $(r,\omega,x_N)$ the cylindrical coordinates of $x$, where $r^2=x_1^2+...+x_{N-1}^2$ and $\omega \in S^{N-2}$, the $(N-2)$-dimensional sphere embedded in the plane of equation $x_N=0$, such that
$
(x_1,\cdots, x_N)=(r\omega, x_N). 
$
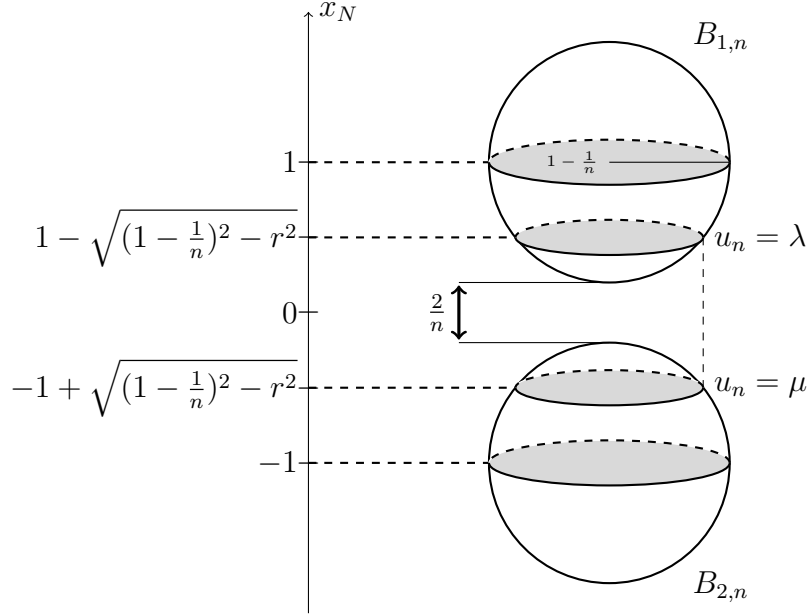
\begin{figure}[h!]
  \centering
 \begin{center}
\begin{tikzpicture}[scale=2]
\draw[->] (0,-2)--(0,2) node[right]{$x_N$};
\draw (0,0) node{$-$} node[left]{$0$};
\draw (0,1) node{$-$} node[left]{$1$};
\draw (0,-1) node{$-$} node[left]{$-1$};
\draw (0,.5) node{$-$} node[left]{$1-\sqrt{(1-\frac{1}{n})^2-r^2}$};
\draw (0,-.5) node{$-$} node[left]{$-1+\sqrt{(1-\frac{1}{n})^2-r^2}$};
\draw[thick] (2,1) circle (.8);
\draw[thick, dashed] (0,1)--(1.2,1);
\draw[thick, dashed] (0,-1)--(1.2,-1);
\draw[thick] (2,-1) circle (.8);
\draw[thick, dashed] (0,.5)--(1.376,.5);
\draw[thick, dashed] (0,-.5)--(1.376,-.5);
\draw (3,2) node[below left]{$B_{1,n}$};
\draw (3,-2) node[above left]{$B_{2,n}$};
\draw[dashed] (2.624,.5) node[right]{$u_n=\lambda$} -- (2.624,-.5) node[right]{$u_n=\mu$};
\draw (2,.2)--(1,.2);
\draw (2,-.2)--(1,-.2);
\draw[very thick,<->] (1,-.18)--(1,.18) node[midway,left]{$\frac{2}{n}$};
\fill[fill=gray!30,draw=white] (2,1) ellipse (.8cm and .15cm);
\fill[fill=gray!30,draw=white] (2,.5) ellipse (.624cm and .117cm);
\fill[fill=gray!30,draw=white] (2,-1) ellipse (.8cm and .15cm);
\fill[fill=gray!30,draw=white] (2,-.5) ellipse (.624cm and .117cm);
\draw (2,1) node[left]{\tiny $1-\frac{1}{n}$}--(2.8,1);
\draw[thick] (1.2,1) arc (180:360:.8 and .15);
\draw[thick,dashed] (2.8,1) arc (0:180:.8 and .15);
\draw[thick] (1.2,-1) arc (180:360:.8 and .15);
\draw[thick,dashed] (2.8,-1) arc (0:180:.8 and .15);
\draw[thick] (1.376,.5) arc (180:360:.624 and .117);
\draw[thick, dashed] (2.624,.5) arc (0:180:.624 and .117);
\draw[thick] (1.376,-.5) arc (180:360:.624 and .117);
\draw[thick, dashed] (2.624,-.5) arc (0:180:.624 and .117);
\end{tikzpicture}
\end{center}
  \caption{\footnotesize{$N=3$, after mollifying the function $u$, we have $u_{n}|_{\overline{B_{1,n}}}=\lambda$ and $u_{n}|_{\overline{B_{2,n}}}=\mu$.}}
  \label{fig1}
\end{figure}

We have
\begin{eqnarray*}
\lambda-\mu & = & u_n\left(r,\omega,1-\sqrt{\left(1-\frac{1}{n}\right)^2-r^2}\right)-u_n\left(r,\omega,-1+\sqrt{\left(1-\frac{1}{n}\right)^2-r^2}\right). 
\end{eqnarray*}

Thus, by integrating $u_n$ along a vertical path from the boundary of $B_{1,n}$ to the boundary of $B_{2,n}$, we obtain
\begin{eqnarray*}
\lambda-\mu & = & \int_{-1+\sqrt{\left(1-\frac{1}{n}\right)^2-r^2}}^{1-\sqrt{\left(1-\frac{1}{n}\right)^2-r^2}} \dfrac{\partial u_n}{\partial x_N}(r,\omega,s)ds.
\end{eqnarray*}
Hence, by Cauchy-Schwarz inequality, we have 
\begin{equation}\label{eq1}
|\lambda-\mu| \leq \sqrt{2\left(1-\sqrt{\left(1-\frac{1}{n}\right)^2-r^2}\right)} \left(\int_{-1+\sqrt{\left(1-\frac{1}{n}\right)^2-r^2}}^{1-\sqrt{\left(1-\frac{1}{n}\right)^2-r^2}} \left(\dfrac{\partial u_n}{\partial x_N}(r,\omega,s)\right)^2 ds\right)^{\frac{1}{2}}. \\
\end{equation}
This yields that
$$
\int_{-1+\sqrt{\left(1-\frac{1}{n}\right)^2-r^2}}^{1-\sqrt{\left(1-\frac{1}{n}\right)^2-r^2}} \left(\dfrac{\partial u_n}{\partial x_N}(r,\omega,s)\right)^2 ds \geq \dfrac{(\lambda-\mu)^2}{2\left(1-\sqrt{\left(1-\frac{1}{n}\right)^2-r^2}\right)}.
$$
When we  integrate $\left( \dfrac{\partial u_n}{\partial x_N}\right)^2$, which is 0 on $B_{1,n}$ and $B_{2,n}$, over the set
$$
A_n=\left\{(r,\omega,x_N) \in U, \ 0\leq r\leq 1-\frac{1}{n}, \ -1\leq x_N \leq 1\right\},
$$
we get
\begin{eqnarray*}
\begin{array}{ll}
 &\ds{\iiint_{A_n} \left( \dfrac{\partial u_n}{\partial x_N}\right)^2} \vspace{4pt} \\ 
&=\ds{\int_0^{1-\frac{1}{n}}\left(\int_{S^{n-2}}\left(\int_{-1+\sqrt{(1-\frac{1}{n})^2-r^2}}^{1-\sqrt{\left(1-\frac{1}{n}\right)^2 -r^2}} \left(\dfrac{\partial u_n}{\partial x_N}(r,\omega,s)\right)^2 ds\right)d\omega \right)dr }\vspace{4pt}\\
 &\geq  \ds{ \left|S^{N-2}\right|(\lambda-\mu)^2\int_0^{1-\frac{1}{n}} \dfrac{r^{N-2}}{2\left(1-\sqrt{(1-\frac{1}{n}^2)-r^2}\right)}dr,}
 \end{array}
\end{eqnarray*}
(where we used the notation $\left|S^{N-2}\right|:=\mathcal{L}^{N-2}(S^{N-2})$
for the $N-2$ dimensional Lebesgue measure of $S^{N-2}$).
Moreover, as $\sqrt{x} \geq x$ for all $x \in [0,1]$, we can write
\begin{eqnarray*}
\sqrt{\left(1-\frac{1}{n}\right)^2-r^2} & = & \left(1-\frac{1}{n}\right)\sqrt{1-\left(\frac{nr}{n-1}\right)^2} \\
& \geq & \left(1-\frac{1}{n}\right)\left(1-\left(\frac{nr}{n-1}\right)^2\right) \\
& \geq & \left(1-\frac{1}{n}\right) -\dfrac{nr^2}{n-1}.
\end{eqnarray*}
Finally,
\begin{eqnarray*}
\dfrac{r^{N-2}}{2\left(1-\sqrt{(1-\frac{1}{n}^2)-r^2}\right)} & \geq & 
\dfrac{r^{N-2}}{2\left(1-\left(1-\frac{1}{n}-\frac{nr^2}{n-1}\right)\right)} \\
& \geq & \ds{\dfrac{1}{2}\times\ds\dfrac{r^{N-2}}{\ds{\frac{1}{n}+\frac{nr^2}{n-1}}}}.
\end{eqnarray*}
In the case $N=2$, we obtain
\begin{equation}\label{eq2}
\begin{array}{lll}
\ds{\int_0^{1-\frac{1}{n}}\dfrac{1}{2}\dfrac{r^{N-2}}{\frac{1}{n}+\frac{nr^2}{n-1}} dr} & = &\ds{\int_0^{1-\frac{1}{n}}\dfrac{1}{2}\dfrac{dr}{\frac{1}{n}+\frac{nr^2}{n-1}}} \\
& = &\ds{ \frac{1}{2}\sqrt{n-1}\left[\arctan\left(\dfrac{nr}{\sqrt{n-1}} \right)\right]_0^{1-\frac{1}{n}}} \\
& \ds{\underset{n\to +\infty}{\sim}} &\ds{ \dfrac{\pi \sqrt{n}}{4}}
\end{array}
\end{equation}
and, in the case $N=3$,
\begin{equation}\label{eq3}
\begin{array}{lll}
\ds{\int_0^{1-\frac{1}{n}}\dfrac{1}{2}\dfrac{r^{N-2}}{\frac{1}{n}+\frac{nr^2}{n-1}} dr} & = &\ds{\int_0^{1-\frac{1}{n}}\dfrac{1}{2}\dfrac{rdr}{\frac{1}{n}+\frac{nr^2}{n-1}}} \\
& = & \dfrac{n-1}{4n}\left[\log\left(\frac{1}{n}+\frac{nr^2}{n-1} \right) \right]_0^{1-\frac{1}{n}}  \\
& \underset{n\to +\infty}{\sim} & \dfrac{\log n}{4}.
\end{array}
\end{equation}
In both cases, we end up with
$$\iiint_U |\nabla u_n |^2 \geq \iiint_{A_n} \left(\dfrac{\partial u_n}{\partial x_N} \right)^2 \geq \left | S^{N-2}\right|(\lambda-\mu)^2 \int_0^{1-\frac{1}{n}}\dfrac{1}{2}\left(\dfrac{r^{N-2}}{\frac{1}{n}+\frac{nr^2}{n-1}}\right) dr,$$
and the right-hand side tends to $+\infty$ as $n$ tends to $+\infty$ if $\lambda\neq \mu$ because of (\ref{eq2}) and (\ref{eq3}).
\vskip 0.3 cm

2. In the case $N \geq 4$, let us consider the function $u \in H^1(U)$ verifying condition ($C_{\lambda,\mu}$) for which we have equality in the Cauchy-Schwarz inequality (\ref{eq1}), in the limit $n \to +\infty$. Over the set
$$A=\left\{(r,\omega,x_N), \ 0\leq r\leq 1, \ -1\leq x_N \leq 1, \ r^2+(x_N-1)^2\geq 1, \ r^2+(x_N+1)^2\geq 1 \right\}.$$
This function is of the form
$$
u(r,\omega,x_N)=\dfrac{\lambda+\mu}{2} + \dfrac{\lambda-\mu}{2(1-\sqrt{1-r^2})}x_N.
$$
For such a function, $\dfrac{\partial u}{\partial x_N}$ is square integrable on $A$. Indeed, the previous inequalities are equalities, and we have
$$
\iiint_A \left( \dfrac{\partial u}{\partial x_N}\right)^2 = \left|S^{N-2}\right| (\lambda-\mu)^2\int_0^1 \dfrac{r^{N-2}}{2(1-\sqrt{1-r^2})}dr 
$$
and $\dfrac{r^{N-2}}{2(1-\sqrt{1-r^2})} \underset{r\to 0}{\sim} r^{N-4}$ which is integrable in the neighborhood of $0$ if $N\geq 4$.\\
Moreover, as $u$ is independent of $\omega$, we are left to prove that $\dfrac{\partial u}{\partial r}$ is square integrable on $A$. Indeed,
$$
\dfrac{\partial u }{\partial r} = \dfrac{\lambda -\mu}{2}x_N \times \dfrac{r}{\sqrt{1-r^2}(1-\sqrt{1-r^2})^2}.
$$
Thus,
\begin{eqnarray*}
&&\iiint_A \left( \dfrac{\partial u}{\partial r} \right)^2 \\
 & = & \int_0^1\int_{S^{N-2}}\int_{-1+\sqrt{1-r^2}}^{1-\sqrt{1-r^2}} \dfrac{(\lambda -\mu)^2}{4}s^2 \times \dfrac{r^2}{(1-r^2)(1-\sqrt{1-r^2})^4} r^{N-2}ds d\omega dr \\
& = &\left|S^{N-2})\right|\dfrac{(\lambda -\mu)^2}{6}\int_0^1(1-\sqrt{1-r^2})^3\times \dfrac{r^N}{(1-r^2)(1-\sqrt{1-r^2})^4} dr \\
& = & \left|S^{N-2})\right|\dfrac{(\lambda -\mu)^2}{6} \int_0^1 \dfrac{r^N}{(1-r^2)(1-\sqrt{1-r^2})} dr.
\end{eqnarray*}
We notice here that $\dfrac{r^N}{(1-r^2)(1-\sqrt{1-r^2})} \underset{r\to 0}{\sim} 2r^{N-2}$. Since $N\geq 4,$ it is integrable in the neighborhood of $0$.
Therefore, $\iiint_A \left|\nabla u \right|^2<\infty$, and we can now easily extend $u$ on $U$ in order to have $u \in H^1(U)$. \hfill $\Box$
%%%%%%%%%%%%%%%%%%%%%%%%%%%%%%%%%%%%%%%%%%%%%%%%%%%%%%%%%%%%%%%%%%%%%%%%%%%
\subsubsection*{Proof of Theorem \ref{easyergodic}}
We know, from (\ref{large-advection}), that $\ds{\lim_{M\rightarrow+\infty}\frac{c^{*}_{\Omega,A, Mq,f}(e)}{M}=\max_{w\in\mathcal{I}_{1}^{A}}\frac{\ds\int_C(q\cdot\tilde{e})\,w^2}{\ds\int_C w^2}}.$ 
From Remark \ref{observation}, this limit is strictly positive whenever there exists a first integral $w_{0}\in \mathcal {I}$ (not necessarily in $\mathcal{I}_{1}^{A}$) such that $\int_{C}q\cdot e w_{0}^{2}\neq0.$ As $q$ has two ergodic components $V_{1}$ and $V_{2}$, it follows from Theorem \ref{main} above that any first integral $w\in \mathcal{I}$ must be constant almost everywhere on $V_{1}$ and $V_{2}$. 

First,  let us assume that $h>0$,  i.e. there is a gap between the cylindrical ergodic components $V_{1}$ and $V_{2}$. We can then find $w_{0}$ which is a smooth periodic function over $\Omega$ such that $w_{0}=\lambda$ over $V_{1}$ and $w_{0}=\mu$ over $V_{2}$ for some $\lambda\neq\mu$.  As $q\equiv0$ on $\Omega \setminus (V_{1}\cup V_{2})$, we then get  $w_{0}\in\mathcal {I}$. Moreover, $w_{0}$ satisfies
\begin{equation}\label{leading}
\begin{array}{ll}
\ds\int_C(q\cdot e)w_{0}^2&=\ds{\int_{C\cap V_{1}}(q\cdot e)w_{0}^2+\int_{C\cap V_{2}}(q\cdot e)w_{0}^2}\vspace {5pt} \\
&=\ds{\lambda^{2}\int_{[0,L_{1}]\times D_{1}}q(x)\cdot e~ dx + \mu^{2}\int_{[0,L_{1}]\times D_{2}}q(x)\cdot e~ dx}\vspace {5pt}\\
& =\ds{(\lambda^{2}-\mu^{2})\int_{[0,L_{1}]\times D_{1}}q(x)\cdot e}.
\end{array}
\end{equation}
The last line follows from (\ref{qcylinders}) and a change of variables between $D_{1}$ and $D_{2}$. Having $\lambda\neq \mu$ and $$\int_{[0,L_{1}]\times D_{1}}q(x)\cdot e=\int_{V_{1}\cap C}q\cdot e=\int_{V_{1}\cap C}v(x)\cdot e \neq 0,$$ we get $\int_C(q\cdot e)w_{0}^2\neq 0$ and this finishes the proof of the sufficient condition. 

Let us now turn to the proof of  the other direction of the theorem. We assume that $h=0$ and we let $w\in\mathcal I$ be any first integral of $q$. Using Theorem \ref{main} and the assumption that $V_{1}$ and $V_{2}$ are ergodic components of $q$,  we know that $w=\lambda$ a.e. on $V_{1}$ and $w=\mu$ a.e. on $V_{2}$ for some constants $(\lambda,\mu)\in\R^{2}$. On the other hand, a first integral $w\in\mathcal{I}$ must be at least $H^{1}_{loc}(\Omega)$. Having $N=3$  together with the fact that $h=0$, Theorem \ref{H1-and-dim} then yields that $\lambda=\mu$. Doing the same computation (\ref{leading}) above, with $w_{0}$ replaced by the present $w$,  we get 
$\int_{C}q\cdot e w^{2}=0.$ This holds for any arbitrarily chosen $w\in\mathcal{I}$ and, therefore, $\lim_{M\rightarrow +\infty}\frac{c^{*}_{\Omega, A, Mq, f}}{M}=0$ in the case where $h=0$.
\hfill$\Box$
%%%%%%%%%%%%%%%%%%%%%%%%%%%%%%%%%%%%%%%
\begin{rem}[About $N=4$]\label{n=4}
The result of Theorem \ref{easyergodic} holds true also in the case where $N=4$. That is, when $\Omega=\R^{4}$ or $\Omega=\R\times \omega$ with $\omega\subseteq \R^{3}$ satisfying (\ref{comega}). When $h>0$, we simply do the same as above in the proof of Theorem \ref{easyergodic}. In the case where $h=0$, that is the cylindrical ergodic components $V_{1}$ and $V_{2}$ are in the same direction $e_{1}=(1,0,0,0)$ and tangent to each other, we need to be a bit more careful. We know that any  $w\in \mathcal{I}$, will be equal to a constant $\lambda$ over $V_{1}$ and to a constant $\mu$ over $V_{2}$. Here, $N=4$, we also claim  that the condition $w\in H^{1}_{loc}(\Omega)$ \emph{together with} $h=0$ will lead to $\lambda =\mu$ and then the analogous quantity in (\ref{leading}) will be $0$. In fact, if $\lambda \neq \mu$, then as $w\in H^{1}_{loc}(\R^{4})$ (or $H^{1}_{loc}(\Omega)$), then for almost every $x_1\in \R$, the function $w(x_1,\cdot,\cdot,\cdot)$ must be in $H^{1}_{loc}(\R^{3})$. In particular $w(0,\cdot,\cdot,\cdot)\in H^{1}(K)$ where we chose $K\subset\R^{3}$ to contain parts of $D_{1}$ as well as parts of $D_{2}$. However, over the bounded set $K$, $w(0,\cdot,\cdot,\cdot)$ takes two different values $\lambda$ and $\mu$. Applying Theorem \ref{H1-and-dim} (part 1), we get a contradiction. Therefore, $\lambda=\mu$ and the function $w$ must have the same constant value over $V_{1}$ and $V_{2}$. This finishes the proof in the 4-dimensional case. 
\end{rem}
\subsubsection*{Proof of Theorem \ref{ngeq5}}  
The proof of Thoerem \ref{ngeq5} is similar to that of Theorem \ref{easyergodic}. However, we do not need to assume that $h>0$, here, as the dimension is $N\geq 5$. Indeed,  we need to construct an $H^{1}_{loc}(\R^{N})$ first integral of $q$  which takes a constant value $\lambda$ on $V_{1}:=\R\times D_{1}$ and a different constant value $\mu$ on $V_{2}:=\R\times D_{2}$. This construction is easy when the cylindrical components $V_{1}$ and $V_{2}$ are parallel and not tangent to each other, i.e. when $h>0$. So, we just look at the case where $h=0$. The regularity we need this first integral, call it $u_{0}$, to have is $H^{1}_{loc}(\R^{N})$. One can see that there is the problem of  square-integrability of the $N^{th}$ partial derivative of $u_{0}$ on any compact which contains a part of $V_{1}$ and another part of $V_{2}$. However, this can be resolved by observing that a slice (in the $x_{1}=0$ plane for instance) of a cylindrical domain as $V_{1}=\R\times D_{2}\subseteq \R\times \R^{N-1}$ or $V_{2}$ is an $(N-1)$-dimensional ball (with $N-1\geq 4$). Hence, we can apply part 2 of Theorem \ref{H1-and-dim} and consider the function $u\in H^{1}_{loc}(\R^{N-1})$ which satisfies the desired properties on the slices $D_{1}\subset \R^{N-1}$ and $D_{2}\subset \R^{N-1}$ of $V_{1}$ and $V_{2}$ respectively. That is, $u=\lambda$ on $D_{1}$ and $u=\mu$ on $D_{2}$ with $\lambda\neq \mu$ and $u\in H^{1}_{loc}(\R^{N-1})$. \\
We now define $u_{0}$,  in $N$-variables,  by $u_{0}=\lambda$ on $V_{1}\subset\R^{N}$ and $u_{0}=\mu$ on $V_{2}\subset \R^{N}$, with $\lambda\neq \mu$. Obviously, the function $u_{0}\in H^{1}_{loc}(\R^{N})$  and hence a first integral of $q$ even though $h=0.$ Now, we redo the same computations (\ref{leading}) with the function $u_{0}\in H^{1}_{loc}(\R^{N})$  instead of $w_{0}$ and get 
$$\int_{C}(q\cdot e) u_{0}^{2}=(\lambda^{2}-\mu^{2})\int_{C}v(x)\cdot e \neq 0\text{, as }\lambda\neq \mu.$$ This leads us to the conclusion that
$$\ds{\lim_{M\rightarrow+\infty}\frac{c^{*}_{\Omega,A, Mq,f}(e)}{M}=\max_{w\in\mathcal{I}}\frac{\ds\int_C(q\cdot\tilde{e})\,w^2}{\ds\int_C w^2}}>0,$$ independently of whether $h$ is positive or zero, whenever $N\geq 5$,  and finishes the proof of Theorem \ref{ngeq5}.
\hfill $\Box$

%%%%%%%%%%%%%%%%%%%%%%%%%%%%%%
\vskip 0.3 cm
\noindent We end this section by giving a short proof of Theorem \ref{shearthm} for the sake of completeness.

\subsubsection*{Proof of Theorem \ref{shearthm}}
 We know from \cite{ek1} and \cite{zlatosARMA} (this is Theorem \ref{ek1} which we reviewed above) that 
$$\lim_{M\rightarrow+\infty}\ds{\frac{\ds{c^{*}_{\Omega,A,M\,q,f}(e)}}{M}}=\ds\max_{\ds{w\in \mathcal{I}_1^A}}\frac{\ds\int_C(q\cdot\tilde{e})\,w^2}{\ds\int_C w^2}= \max_{\ds{w\in \mathcal{I}}}\frac{\ds\int_C(q\cdot\tilde{e})\,w^2}{\ds\int_C w^2},$$ where $C$ is the periodicity cell of $\Omega$. Notice that, in this particular case, $C=[0,L_{1}]\times C_{\omega}$ where $L_{1}$ is the $x_{1}$ period of the diffusion $A$ and the reaction $f$ in equation (\ref{heteqM}) and $C_{\omega}$ is the periodicity cell of the section $\omega$ of $\Omega$ (in the case where $d\geq 1$) and $C_{\omega}=\omega$ in the case $d=0$.  The reason of being able to write the limit as a  maximum over the set $\mathcal{I}$ was given in Remark \ref{observation} above. In this present situation, the vector field $q$ is uni-directional. That is, for any $x\in\Omega$ where $q_{1}(x)\neq 0$, the streamline $T_{x}$ passing through $x$ is parallel to $e.$ As $q_{1}\not\equiv 0$ in $\omega$ and $\int_{C_{\omega}}q_{1}=(\int_{C}q_{1})/L_{1}=0$, there exists $z\in\omega $ and an open neighborhood  $U\subseteq \omega$ of $z$  such that $q_{1}(z)> 0$ and $q_{1}>0$ on $U$. Let $\alpha=\alpha(x_{2},\cdots,x_{N})\in C^{\infty}_{c}(U)$ be a compactly supported smooth function such that $\alpha(z)>0$, $
\alpha \geq 0$, $\alpha\not \equiv 0$ in $U$. Thus, $\int_{U}q_{1}\alpha^{2}>0$. We may, without any loss, normalize $\alpha$ so that $\int_{C}\alpha^{2}dx =1$.

 Moreover, as $q=(q_{1}(x_{2},\cdots,x_{N}),0,
 \cdots,0)$, then every first integral $w$ of $q$ is independent of $x_{1}$ and has the form $w=\beta(x_{2},\cdots,x_{N})$. Therefore, $\mathcal {I}=
 \mathcal{J}$ in the case of shear flows. Hence, the function $\alpha \in \mathcal{I}$ and the following holds $$\max_{\ds{w\in \mathcal{I}}}\frac{\ds\int_C(q\cdot\tilde{e})\,w^2}{\ds\int_C w^2}\geq \int_{C} q_{1}\alpha^{2 }dx=L_{1}\int_{C_{\omega}}q_{1}\alpha^{2}dx_{2}\cdots dx_{N}>0,$$
  which completes the proof of Theorem \ref{shearthm}.\hfill $\Box$

\appendix
\section{Incompressible flows carry no ergodic components in 2D}\label{No-ergodicityin2d}
 \emph{In the two-dimensional case, there are no incompressible flows with ergodic components in the sense of Definition \ref{ergodic-comp}}. 
 
 Indeed, if $q \in C^{1,\delta}(\overline{\Omega})$ is a vector field verifying (\ref{cq}) and (\ref{cq2}), with $\Omega\subseteq\R^{2}$, there exists $\phi \in C^{2,\delta}(\overline{\Omega})$ such that $q=\nabla^{\perp} \phi$. The function $\phi$ is then a \emph{first integral of} $q$ and thus is constant a.e on any ergodic component by Theorem \ref{main}.
Assume, to the contrary, that there exists  $E$ an ergodic component of $q$ and $\lambda \in \R$ such that $\phi(x)=\lambda$ a.e. on $E$. We assume moreover that $q$ does not vanish on $E$, since any stationary point of $q$ is already stable by the flow. We finally set $\tilde{E}=\{ x \in E, \ \phi(x)=\lambda \text{ and } q(x)\neq0\}$.\\
$\tilde{E}$ is clearly an ergodic component of $q$, has same Lebesgue measure as $E$, and since $\tilde{E}$ does not contain any stationary point of $q$, we have
$$
\tilde{E} \subset \partial V_\lambda, \ \text{ where } V_\lambda=\{ x \in \Omega, \ \phi(x)<\lambda\}.
$$
Since the outward unit normal of $V_{\lambda}$ is $\textbf{n}=\dfrac{\nabla\phi}{|\nabla\phi|}$ whenever it is defined, it then follows, from Stokes theorem, that
\begin{equation}\label{levelset}
\int_{\tilde{E}}|q| \leq \oint_{\partial V_{\lambda}} |q|= \oint_{\partial V_{\lambda}} \nabla \phi \cdot \textbf{n} = \iint_{V_{\lambda}} \Delta \phi <+\infty.
\end{equation}
In fact, (\ref{levelset}) is also true for each of the subsets $$\tilde{E}_{k}=\left\{x\in E,~\phi(x)=\lambda\text{ and }|q(x)|\geq\frac{1}{k}\right\},~k\in\mathbb{N}$$
of $\tilde{E}.$
That is, 
$$\forall k\in\mathbb{N}, ~\frac{1}{k}\mathcal{L}^{1}(\tilde{E}_{k})\leq\iint_{V_{\lambda}}\Delta \phi<\infty.$$
Hence, $\mathcal{L}^2(\tilde{E})=0$ because $\tilde{E}$ is $\sigma$-finite for the 1-dimensional Lebesgue measure ($\ds{\tilde{E}=\cup_{k\in\mathbb{N}}\tilde{E}_{k}}$). This is a contradiction and so an ergodic component can not exist in the two-dimensional case.

\section{The zero-average assumption on incompressible flows}
We claimed  in Remark \ref{remarkonzeroaverage} above that the normalization  (\ref{cq2}) of the incompressible vector field is equivalent to having all the $N$ (not only $d$) components of the advection field $q$. This will be the goal of the following proposition.
\noindent As the coefficients and the domain of the reaction-advection-diffusion which we consider have a periodic structure, it will be sometimes convenient to use the following notations.
\begin{defin}Having a domain $\Omega\subseteq\R^{d}\times\R^{N-d}$ with a periodic nature given by (\ref{comega}), we denote the set of equivalence classes modulo the periods $L_{1},\cdots, L_{d}$ by 
\begin{equation}\label{omegahat}
\ds \hat{\Omega}:= \Omega\Big/{L_{1}\mathbb{Z}\times\cdots\times L_{d}}\mathbb{Z}\times\{0\}^{N-d}=\Omega\Big/L_{1}e_{1}\oplus\cdots\oplus L_{d}e_{d},
\end{equation}
where $\{e_{1},\cdots, e_{d},\cdots,e_{N}\}$ is the standard basis of $\R^{N}=\mathbb{R}^{d}\times\R^{N-d}.$
We also set 
\begin{equation}
\ds T:= \R^{N}\Big/{L_{1}\mathbb{Z}\times\cdots\times L_{d}}\mathbb{Z}\times\{0\}^{N-d}=\R^{N}\Big/L_{1}e_{1}\oplus\cdots\oplus L_{d}e_{d} \text{ .}
\end{equation}
Each $x\in \Omega$ will then have an equivalence class $\hat x\in \hat\Omega$ and for each $L$-periodic function $h:\Omega\rightarrow\R^{m}$ ($m\in\mathbb{N}$) we can define the function $\hat h:\hat\Omega\rightarrow\R^{m}$ defined by $$\forall x\in \Omega, \hat h(\hat x):=h(x).$$
\end{defin}

\begin{proposition}\label{zeroaverageproperty}
Consider an incompressible field $q:\Omega\rightarrow \mathbb{R}^{N}$ satisfying (\ref{cq}) and (\ref{cq2}). Then, 
$$\forall \, 1\leq i\leq N, ~~\int_{C}q_{i}(x)dx=0$$ or equivalently $\ds {\int_{C}q(x)dx=0},$ where $C$ is the periodicity cell of $\Omega$ (see (\ref{periodicityCell}) above.)
\end{proposition}
\noindent The proof of  Proposition \ref{zeroaverageproperty} is an application of the following preliminary lemma.
\begin{lem}\label{keylemma}
Let $w\in \mathcal {I}$ be a first integral of $q$. Then, for all $\phi\in H^{1}_{per, loc}(\Omega)\equiv H^{1}(\hat \Omega),$ we have 
\begin{equation}\label{key1}
\int_{C}wq\cdot\nabla\phi ~dx=0=\int_{\ds \hat\Omega}\hat w \hat q\cdot\nabla\hat\phi.
\end{equation}

\end{lem}

\begin{proof} Since $\nabla \cdot q\equiv 0$ in $\Omega,$ it follows that for
all $\hat \phi\in H^{1}(\hat\Omega)$ and for almost every $\hat x\in \hat \Omega$  $$\nabla\cdot(\hat q(\hx) \hat\phi(\hx))=\hq(\hx)\cdot\nabla\hat{\phi}(\hx).$$ Due to the periodicity of $\phi,$ $w,~q$ and $\Omega$ and the condition $q\cdot\nu=0$ on $\partial\Omega$, the boundary terms vanish in the following integrals (integrating by parts):
$$\begin{array}{lll}
\ds \int_{C}wq\cdot\nabla\phi ~dx &:=\ds \int_{\ds \hat\Omega}\hat w \hat q\cdot\nabla\hat\phi ~d\hx&=\ds \int_{\ds\hat\Omega}\hat w \nabla\cdot(\hat q\hat\phi )~d\hx\vspace{5 pt}\\
&&=-\ds \int_{\ds\hat\Omega} \nabla \hat w \cdot \hat q\hat\phi \vspace{5 pt} \\
&&=0\text{ (since $w\in \mathcal{I}$.)}
\end{array}$$
\end{proof}

\noindent \textbf{Proof of Proposition \ref{zeroaverageproperty}}.  We already have $\int_{C}q_{i}=0$ for all $i\leq d.$ Fix any $i$ such that $N\geq i>d$ and observe that $q_{i}=q\cdot e_{i},$ where $e_{i}$ is the ith member of the canonical basis. On the other hand, we have $e_{i}= \nabla_{x} \,x_{i}$ and, as $i>d,$ $\phi(x)=x_{i}$ is $(L_{1},\cdots ,L_{d})$-periodic in $(x_{1},\cdots,x_{d})$ over $\Omega.$ Thus, $\phi(x)=x_{i}$ is admissible as a test function in (\ref{key1}). We apply Lemma \ref{keylemma} with $\phi(x)=x_{i}$ and $w\equiv1$ to get
$$\int_{C}q\cdot e_{i}=0.$$
This proves that $\ds{\int_{C}q_{i}(x)dx =0}$ for all $i>d$ and completes the proof of Proposition \ref{zeroaverageproperty}.\hfill $\Box$

In the precise limit which describes the asymptotic behavior of the KPP minimal speed within large advection, \cite{ek1,ekcras,zlatosARMA} show that the quantity $\int_{C}q\cdot e w^{2}$ where $e$ is a fixed unitary direction in $\R^{d}\times\{0\}^{N-d}$ and $w\in\mathcal I$ plays the important role. To analyze this variational quantity over the family of first integrals $\mathcal{I}$ in dimensions $N\geq 3$, the next lemma could  be seen as a preliminary tool to work with ``$N-1$ dimensional slices'' of the cell $C\subseteq\Omega$.
\begin{lem}[quantities over a slice of the domain] \label{over-a-slice-lemma}
Let $w\in\mathcal I$ ($w$ is a first integral of $q$) where $q$ satisfies (\ref{cq}) and (\ref{cq2})  and the domain $\Omega\subseteq \R^{N}$ satisfies (\ref{comega}). Fix any  $1\leq i \leq d$ and denote by
\begin{equation}\label{slice}
\begin{array}{lll}
\ds{C_{i, 0}}&:=&\ds{C\cap \ds{\left\{x_{i}=0\right\}}}\\
&=&\left\{(x_{1},\cdots,x_{d},x_{d+1},\cdots,x_{N})\in C \subset \Omega\text{ such that } x_{i}=0\right\}.
\end{array}
\end{equation}
Then, 
\begin{equation}\label{integral-over-slice}
\int_{C}(q\cdot e_{i})w^{2} dx=L_{i}\ds{\int_{\ds{C_{i,0}}}(q\cdot e_{i}) w^{2}}.
\end{equation}
\end{lem}
\begin{rem} The result in the above lemma can be stated also for integrals of the form $\int_{C}(q\cdot e_{i})w^{m}dx$ for all $m\geq 1$--as long as $w^{m}\in{H^{1}_{loc}}(\Omega)$ and $q\cdot \nabla w=0$ a.e. in $\Omega$. We will demonstrate Lemma \ref{over-a-slice-lemma} only in the case $m=2$. 
\end{rem}

\begin{proof} The proof of this lemma is similar to that of Lemma \ref{keylemma} above. The only difference here is that we consider here a function $\phi\in H^{1}_{loc}(\Omega)$ instead of $H^{1}(\ho).$ This means that the test functions $\phi$ in the present proof are not necessarily periodic, and hence, we may have non-zero boundary terms upon integrating by parts.  Indeed, since $w$ is a first integral of $q,$ we then have
$\phi q\cdot \nabla (w^{2})=0$ for all $\phi\in H^{1}_{loc}(\Omega)$ and for almost every $x\in \Omega\supseteq C$. This, together with $\nabla \cdot q\equiv0$ in $\Omega$, yields that
\begin{equation}\label{computation}
\begin{array}{rll}
0=&\ds{\int_{C}\phi q\cdot \nabla (w^{2})}\vspace{3 pt}\\
=& \ds{\int_{C}\nabla\cdot(\phi qw^{2})-\int_{C}\nabla\phi\cdot qw^{2}} \vspace{3 pt}\\
=&\ds{\int_{\partial C}\phi w^{2}q\cdot \mathbf{n} d\sigma-\int_{C}\nabla \phi\cdot qw^{2}},
\end{array}
\end{equation}
where $\mathbf{n}$ stands for the outward unit normal vector to the cell $C$ and $\sigma$ is the Lebesgue measure induced  over $\partial C$.  %Now, we fix any real number $a\in\R$ and, for convenience, we set 
%$a e_{i}+C$  the subset of $\Omega$ which is produced upon shifting the %periodicity cell $C$ by $a$ in the $i$th coordinate. That is,
%$$a e_{i}+C=\{x+ae_{i}\text{, such that }x\in C\}\subset \Omega.$$ Also, for any %%function $h:\Omega\rightarrow\R^{m}$ ($m\geq1$), we associate  $h_{a}:%%%\Omega\rightarrow\R^{m}$ as follows
%%%$$\forall x\in\Omega, ~~h_{a}(x):=h(x-ae_{i}).$$ 
%%Again, as $\nabla\cdot q\equiv 0$ in $\Omega$, the above computation in %%%%%(\ref{computation}) can be generalized to 
%%%$$\begin{array}{rll}
%%%0=&\ds{\int_{ae_{i}+C}\phi_{a}(x) q_{a}(x)\cdot \nabla (w_{a}(x))^{2})dx}\vspace{3 pt}\\
%%=&\ds{\int_{\partial C}\phi_{a} w_{a}^{2}q_{a}\cdot \mathbf{n}_{a} -%%%%\int_{ae_{i}+C}\nabla \phi_{a}\cdot q_{a}w_{a}^{2}},
%%\end{array}
%%%$$  
Since $i\leq d$,  $\Omega$ is $L_{i}$-periodic in the ith direction and thus we can rewrite $\partial C$ in the following format
\begin{center} $\partial C=(C\cap\partial \Omega)\cup C_{i,0}\cup C_{i,L_i},$\end{center}
where 
$$\forall \kappa\in [0,L_{i}], ~C_{i,\kappa}:=\{x\in C \text{ such that } x_{i}=\kappa\}.$$ 
%%%%%%%%%%%%%%%%%%%%%%%%%%%%%%%%%%%%%%%
Notice that, in this setting, we have $\mathbf{n}\equiv \nu$ on $C\cap \partial \Omega$ where we can use the assumption $q\cdot\nu\equiv 0$; while,  $\mathbf{n}= -e_{i}$ on $C_{i,0}$ and  $\mathbf{n}= e_{i}$ on $C_{i,L_i}$.  
As a consequence, we can rewrite the last computation as 
\begin{equation}\label{essential}
\begin{array}{rll}
0=&\ds{\int_{\partial C}\phi w^{2}q\cdot \mathbf{n} -\int_{C}\nabla \phi\cdot qw^{2}}\vspace{4 pt}\\
=&\ds{\int_{C\cap\partial\Omega}\phi w^{2}\ds\underbrace{q\cdot \mathbf{n}}_{=0} \,-\int_{C_{i,0}}\phi w^{2}q\cdot e_{i}+\int_{C_{i,L_i}}\phi w^{2}q\cdot e_{i}}\vspace{4pt}\\
&\ds{-\int_{C}\nabla \phi\cdot qw^{2}}\vspace{4pt}\\
=&\ds{-\int_{C_{i,0}}\phi w^{2}q\cdot e_{i}+\int_{C_{i,L_{i}}}\phi w^{2}q\cdot e_{i}}-\ds{\int_{C}\nabla \phi\cdot qw^{2}}
\end{array}
\end{equation}

In this step, we pick the test function $\phi(x)=x_{i}$ for all $x\in \Omega.$ Notice that $\phi$ is \emph{not} $L_{i}$-periodic in $x_{i}$, while $\nabla \phi=e_{i}$, $\phi|_{C_{i,L_{i}}}=L_{i}$ and $\phi|_{C_{i,0}}=\ds{x_{i}|_{\{x_{i}=0\}}}=0$. With this particular $\phi$, $\ds{\int_{C}\nabla \phi\cdot qw^{2}}$ can be written as 
$\ds{\int_{C}e_{i}\cdot q w^{2}}.$ We can also simplify the other  terms of the right hand side of (\ref{essential}) as follows 
$$\begin{array}{ll}
\ds{-\int_{C_{i,0}}\phi w^{2}q\cdot e_{i}+\ds\int_{{C_{i,L_i}}}\phi w^{2}q\cdot e_{i}}
&=\ds{-\int_{C_{i,a}} 0\times (w^{2}q\cdot e_{i})+L_{i}\int_{C_{i,L_i}} w^{2}q\cdot e_{i}}\vspace{4 pt}\\
&=\ds{L_{i}\int_{C_{i,0}} w^{2}q\cdot e_{i}}\text{ (as $q$ is $L_{i}$-periodic.)}
\end{array}$$ 
Eventually, we get from (\ref{essential}) and the last simplification
\begin{equation}\label{integral-over-slice-0}
0=L_{i}\int_{C_{i,0}} w^{2}q\cdot e_{i}-\int_{C}e_{i}\cdot qw^{2}.
\end{equation}
and this completes the proof of Lemma \ref{over-a-slice-lemma}.
\end{proof}

\section*{Acknowledgements}
The authors wish to thank professors  Mary Pugh, Robert Jerrard, Borris Khesin and 
Konstantin Khanin, at the University of Toronto, for the enlightening discussions during this research project. M.E. is grateful to NSERC-Canada for providing support under the NSERC postdoctoral fellowship 6790-403487.

\end{document}